\documentclass{article}
\usepackage[margin=3cm]{geometry}
\usepackage{inputenc}
\usepackage{float}
\usepackage{amssymb}
\usepackage{amsmath,amsfonts,amsthm}
\usepackage{ wasysym }
\newtheorem{thm}{Theorem}
\usepackage{mathtools}
\usepackage{hyperref}
\usepackage{enumerate}
\usepackage{array}
\usepackage{comment}

\DeclarePairedDelimiter\floor{\lfloor}{\rfloor}
\usepackage{color}
\usepackage{tikz}
\usetikzlibrary{matrix,arrows,decorations.pathmorphing}
\usepackage{titlesec}

\title{Recovering Conductances of Resistor Networks in a Punctured Disk}

\author{Yulia Alexandr\thanks{Wesleyan University, CT} \and Brian Burks\thanks{University of California, Berkeley, CA}\and Sunita Chepuri\thanks{University of Minnesota, Twin Cities, MN} \and Patricia Commins\thanks{Carleton College, MN}}

\begin{document}

\maketitle

\theoremstyle{definition}
\newtheorem{exc}{Exercise}
\newtheorem{prob}{Problem}
\newtheorem{lemma}{Lemma}
\newtheorem{ex}{Example}
\newtheorem{cor}{Corollary}
\newtheorem{rem}{Remark}

\newtheorem{defn}{Defnition}

\newtheorem{conj}{Conjecture}

\newcommand{\QQ}{\mathbb{Q}}
\newcommand{\ZZ}{\mathbb{Z}}
\newcommand{\NN}{\mathbb{N}}
\newcommand{\RR}{\mathbb{R}}
\newcommand{\CC}{\mathbb{C}}
\begin{abstract}
    \noindent The response matrix of a resistor network is the linear map from the potential at the boundary vertices to the net current at the boundary vertices.  For circular planar resistor networks, Curtis, Ingerman, and Morrow have given a necessary and sufficient condition for recovering the conductance of each edge in the network uniquely from the response matrix using local moves and medial graphs. We generalize their results for resistor networks on a punctured disk. First we discuss additional local moves that occur in our setting, prove several results about medial graphs of resistor networks on a punctured disk, and define the notion of $z$-sequences for such graphs.  We then define certain circular planar graphs that are electrically equivalent to standard graphs and turn them into networks on a punctured disk by adding a boundary vertex in the middle.  We prove such networks are recoverable and are able to generalize this result to a much broader family of networks. A necessary condition for recoverability is also introduced.
    
\end{abstract}

\section{Introduction}

In this paper, we study resistor networks, electrical networks made up of only resistors.  The electrical properties of these networks is characterized by the response matrix.  The response matrix gives the linear map from a potential assignment at each boundary vertex to the net current flow at each vertex.  A common question in the study of electrical networks is when we can uniquely recover the conductances of a resistor network given the response matrix.  This is called the Dirichlet-to-Neumann problem.

This problem has been extensively studied for a class of networks called circular planar resistor networks.  Curtis, Moores, and Morrow \cite{curtis_mooers_morrow} developed an algorithm to recover the conductances of a particular family of these graphs called standard graphs.  Curtis, Ingerman, and Morrow  \cite{curtis_ingerman_morrow} and De Verdiere, Gitler, and Vertigan \cite{deverdiere_gitler_vertigan} were able to use this algorithm to obtain results for all circular planar resistor networks. They considered the property of criticality, a condition regarding the connections of a network.  Building on the previous algorithm and using medial graphs, they proved that the networks where we uniquely recover conductances from the response matrix are exactly the critical networks.  They were also able to use medial graphs to show a complete catalogue of local moves that preserve the response matrix of a circular planar resistor network and use these local moves to classify critical networks.  Moreover, they were able to identify all response matrices of resistor networks with a positivity criteria for minors of the matrix.

Kenyon and Wilson \cite{KW1,KW2} studied the combinatorics of the response matrix for resistor networks.  In particular, they were able to realize the minors of the response matrix for a circular planar resistor network as a sum over groves in the network.  This explains the positivity criteria for response matrices of circular planar resistor networks found in earlier work.

However, very little is known for more complicated networks.  In \cite{2011arXiv1104.4998L}, Lam and Pylyavskyy studied the Dirichlet-to-Neumann 
problem for electrical networks on a cylinder. They gave a conjectural solution for general cylindrical electrical networks, and showed it held for a special class of these networks known as ``purely cylyndrical'' networks.

In this paper, we introduce resistor networks in a punctured disk (rnpds). These are networks embedded in a disk where exactly one boundary vertex is placed inside the disk. We first discuss new local moves that arise for these networks.  We conjecture that the local moves stated in the paper are the only moves that preserve electrical equivalence.  Next we define medial graphs for rnpds and prove several properties of the medial graph for irreducible rnpds.  Medial graphs allow us to define spider graphs, a famliy of rnpds closely related to the standard graphs.  Using spider graphs, we find several conditional local moves that can be applied to rnpds under certain conditions without changing the response matrix.  Turning to the question of recoverability, we find a class of recoverable rnpds.  We also give a necessary condition for recoverability or rnpds.

\section{Background and Definitions} \label{backgroundanddefinitions}
For more background on electrical networks, see \cite{curtis_ingerman_morrow}, \cite{deverdiere_gitler_vertigan}, and \cite{2012arXiv1203.1256K}.
\begin{defn}
A \textit{resistor network} is a graph $(V, E)$ with a specified set $B \subseteq V$ of \textit{boundary vertices} and a real nonnegative conductance $c_e$, for each $e \in E.$ The remaining vertices, $I = V \setminus B$, are called \textit{internal vertices}.  We will say graph to refer to a resistor network without conductances.
\end{defn}

In this paper, we'll assume all resistor networks are connected graphs. Our convention throughout this paper will be to color boundary vertices white and internal vertices black.
\begin{defn}
A \textit{circular planar resistor network} (cprn) is a resistor network that can be embedded in a disk so that it is planar and all the boundary vertices are on the boundary of the disk. 
A \textit{resistor network in a punctured disk} (rnpd) is a resistor network that can be embedded in a disk so that it is planar and all boundary vertices but one are on the boundary of the disk.  See Figure~\ref{fig:cprn-rnpd}.
\end{defn}

\begin{figure}[h]
    \label{fig:example}
  \begin{minipage}[t]{0.5\textwidth}
  \begin{center}
\begin{tikzpicture}[scale=0.8]
          \draw [line width = 0.25mm, fill = black] (0, 2)edge node[right]{$2$}(2,0);
          \draw [line width = 0.25mm, fill = black] (0, 2)edge node[left]{$1$}(-2,0);
          \draw [line width = 0.25mm, fill = black] (0, -2)edge node[below]{$1$}(2,0);
          \draw [line width = 0.25mm, fill = black] (0, -2)edge node[left]{$3$}(-2,0);
          \draw [line width = 0.25mm, fill = black]  (2,0)edge node[above]{$1$}(-2,0) ;
          \draw (2,0) edge[bend right] node[above]{$1$} (0,-2);
          
          \draw [line width=0.25mm, fill=black] (0, 2) circle (1.5mm);
          \draw [line width=0.25mm, fill=black] (0, -2) circle (1.5mm);
          \draw [line width=0.25mm, fill=white] (-2,0) circle (1.5mm);
          \draw [line width=0.25mm, fill=white] (2,0) circle (1.5mm);
    \end{tikzpicture}
      \end{center}
\end{minipage}
\begin{minipage}[t]{0.5\textwidth}
  \begin{center}
\begin{tikzpicture}[scale=0.8]

          \draw [line width = 0.25mm, fill = black] (0, 2)edge node[above]{$1$}(2,0);
          \draw [line width = 0.25mm, fill = black] (0, 2)edge node[above]{$2$}(-2,0);
          \draw [line width = 0.25mm, fill = black] (0, -2)edge node[below]{$7$}(2,0);
          \draw [line width = 0.25mm, fill = black] (0, -2)edge node[below]{$3$}(-2,0);
          \draw [line width = 0.25mm, fill = black]  (2,0)edge node[above]{$1$}(0,0) ;
          \draw [line width = 0.25mm, fill = black]  (-2,0)edge node[above]{$1$}(0,0) ;
          \draw [line width = 0.25mm, fill = black]  (0,0)edge node[left]{$5$}(0,-2) ;
          \draw [line width = 0.25mm, fill = black]  (0,0)edge node[left]{$1$}(0,2) ;

          \draw [line width=0.25mm, fill=black] (0, 2) circle (1.5mm);
          \draw [line width=0.25mm, fill=black] (0, -2) circle (1.5mm);
          \draw [line width=0.25mm, fill=white] (-2,0) circle (1.5mm);
          \draw [line width=0.25mm, fill=white] (2,0) circle (1.5mm);
          \draw [line width=0.25mm, fill=white] (0,0) circle (1.5mm);
    \end{tikzpicture}
      \end{center}
\end{minipage}
\caption{A cprn (left) and rnpd (right)}
\label{fig:cprn-rnpd}
\end{figure}
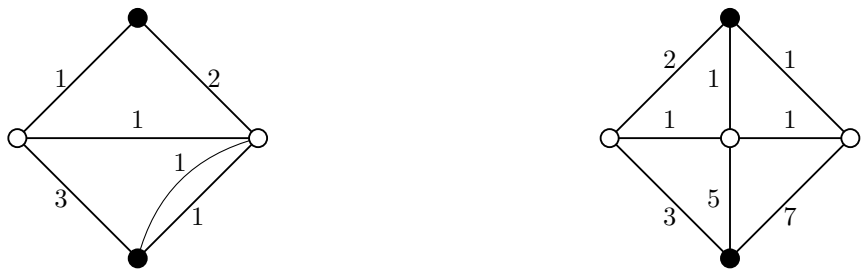
\begin{defn}
A \textit{non-trival rnpd} is an rnpd that cannot be drawn as a cprn. Note that in a non-trivial rnpd, the interior boundary vertex of must be surrounded by a polygon. Otherwise, we are able to redraw the network so that the interior boundary vertex is on the boundary of the disk.
\end{defn}

\begin{defn}
A \textit{potential function} for a resistor network $(V, E)$ is a function $f:V\to\mathbb{R}_{\geq0}$. An edge $(v_1,v_2)$ has a \emph{voltage} given by the difference in potential at $v_1$ and $v_2$.
\end{defn}

We want our resistor networks to model electrical circuits.  This means they need to satisfy Ohm's Law and Kirchhoff's Current Law.

\begin{thm}[Ohm's Law] If $v_e$ is the voltage across edge $e$, then $v_ec_e=i_e$, where $i_e$ is the current across edge $e$.
\end{thm}

By Ohm's Law, a potential function induces a current across each edge. Note that current always flows from the vertex with the higher potential to the vertex with the lower potential.

\begin{thm}[Kirchhoff's Current Law] For each node in an electrical circuit, the sum of the currents entering the node must be the same as the sum of the currents leaving the node.
\end{thm}

The analogue of Kirchhoff's Current Law for resistor networks is that for each internal vertex in a resistor network, the sum of the currents entering the vertex must be the same as the sum of the currents leaving the vertex. If we use Ohm's Law to calculate current, only certain potential functions satisfy Kirchhoff's current Law.

\begin{thm}[\cite{curtis_ingerman_morrow}]
Given a function $\phi:B\to\mathbb{R}_{\geq0}$ for a resistor network, there is a unique choice of potential function $f$ where $f|_B=\phi$ and the currents induced by $f$ from Ohm's Law satisfy Kirchhoff's Current Law.
\end{thm}

\begin{defn}
Let $\Gamma=(V,E)$ be a resistor network with $m$ vertices, $n$ of which are boundary vertices. Index all boundary vertices $v\in B$ as $\{1,\cdots,n\}$ and all internal vertices $v\in I$ as $\{n+1,\cdots,m\}$. Then, we define the \textit{Kirchhoff matrix} of $\Gamma$ as follows:

$$K_{i,j}=
\begin{cases}
\displaystyle \sum_{e=(i,j)}c_e & i\neq j\\
\displaystyle -\sum_{\substack {e= (i,k)\\ k\in[n]}} c_e & i=j\\
\end{cases}$$
\end{defn}

Note that the Kirchhoff matrix can be represented as:
$$K=\begin{bmatrix}
    A & B \\
    B^T & C
\end{bmatrix}$$
where $A$ is a symmetric $n\times n$ matrix, $B$ is an $n\times (m-n)$, and $C$ is a symmetric $(m-n)\times(m-n)$ matrix. 

\begin{defn}
The \textit{response matrix}, $\Lambda(\Gamma)$ of a resistor network $\Gamma$ is the Schur complement $A-BC^{-1}B^T$ of the Kirchhoff matrix.
\end{defn}

We can represent a potential assignment to the boundary vertices as a vector $u\in\mathbb{R}_{\geq0}^n$.  $\Lambda(\Gamma)$ gives the linear map from such an assignment to the vector of currents flowing into each boundary vertex (where a negative value represents current flowing out of the vertex).

\newpage
\begin{ex} \label{cprnexample}
Consider the following cprn:
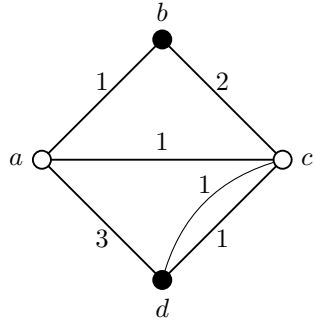
\begin{center}
          \begin{tikzpicture}[scale=0.8]
          \draw [line width = 0.25mm, fill = black] (0, 2)edge node[above]{$2$}(2,0);
          \draw [line width = 0.25mm, fill = black] (0, 2)edge node[above]{$1$}(-2,0);
          \draw [line width = 0.25mm, fill = black] (0, -2) edge node[below]{$1$}(2,0);
          \draw [line width = 0.25mm, fill = black] (0, -2)edge node[below]{$3$}(-2,0);
          \draw [line width = 0.25mm, fill = black]  (2,0)edge node[above]{$1$}(-2,0) ;
          \draw (2,0) edge[bend right] node[above]{$1$} (0,-2);
          
          \draw [line width=0.25mm, fill=black] (0, 2) circle (1.5mm);
          \draw [line width=0.25mm, fill=black] (0, -2) circle (1.5mm);
          \draw [line width=0.25mm, fill=white] (-2,0) circle (1.5mm);
          \draw [line width=0.25mm, fill=white] (2,0) circle (1.5mm);
          
          \node at (0,2) [label=above:$b$]{};
          \node at (0,-2) [label=below:$d$]{};
          \node at (2,0) [label=right:$c$]{};
          \node at (-2,0) [label=left:$a$]{};
    \end{tikzpicture}
\end{center}
We see that the vertex set is $V=\{a,b,c,d\}$. The set of boundary vertices is $B=\{a,c\}$ and the set of internal vertices is $I=\{b,d\}$. Thus, in order to compute the Kirchhoff matrix, we label $a$ and $c$ by 1 and 2, respectively, and $b$ and $d$ by 3
and 4, respectively. Then, the Kirchhoff matrix is:

$$K=\begin{bmatrix}
   -5 & 1 & 1 & 3 \\
    1 & -5 & 2 & 2 \\
    1 & 2 & -3 & 0 \\
    3 & 2 & 0 & -5 \\
\end{bmatrix}$$

Now, as described above, one can compute the response matrix:

$$\begin{bmatrix}
   -5 & 1 \\
    1 & -5 \\
    \end{bmatrix} -  \begin{bmatrix}
   1 & 3 \\
    2 & -2 \\
    \end{bmatrix} \begin{bmatrix}
   -3 & 0 \\
    0 & -5 \\
    \end{bmatrix}^{-1}\begin{bmatrix}
   1 & 3 \\
    2 & -2 \\
    \end{bmatrix}^T=\begin{bmatrix}-\frac{43}{15}&\frac{7}{15}\\ \frac{7}{15}&-\frac{43}{15}\end{bmatrix}$$ 
\end{ex}

\begin{defn}

Graphs $\Gamma_1$ and $\Gamma_2$ are \textit{electrically equivalent}, if, for every assignment of conductances for $\Gamma_1$, there exists an assignment of conductances to $\Gamma_2$ such that the resulting electrical networks have the same response matrix, and vice versa.
\end{defn}

We are interested in three properties of networks that have been shown to be the same for the cprn case: irreducibility, recoverability, and criticality.

\begin{defn}
We call a graph \textit{reducible} if it is electrically equivalent to a graph with fewer edges, and \textit{irreducible} otherwise. 
\end{defn}

Note that a network may have the same response matrix as a network with fewer edges while still being irreducible.

\begin{defn}
A resistor network is \textit{recoverable} if given its response matrix and and graph, we are able to uniquely determine its conductances.
\end{defn}

\begin{defn}
A \textit{connection} of a resistor network $\Gamma = (V, E)$ is a tuple $(P, Q)$ s.t. $P = \{p_1, \cdots, p_k\}$ and $Q = \{q_1, \cdots, q_k\}$ are subsets of $B$ and there exists $k$ disjoint paths that do not pass through boundary vertices connecting $p_i$ to $q_{\pi(i)}$, $\forall i$ and for some $\pi \in S_k$. For a cprn, we require that $p_1, \cdots, p_k, q_k, \cdots, q_1$ occur in clockwise order around the disk's boundary. For an rnpd, we have the same requirement, with the exception that the interior boundary vertex can appear in either $Q$ or $P$, at any index.
\end{defn}

\begin{rem}
If $\Gamma$ is a cprn and $(P, Q)$ is a connection, then the only possible $\pi \in S_{k}$ connecting $P$ and $Q$ is the identity permutation. Otherwise, the paths will not be disjoint.
\end{rem}

\begin{defn}
A cprn is \textit{critical} if the removal of any edge in the network breaks a connection. 
\end{defn}

Curtis, Ingerman, and Morrow \cite{curtis_ingerman_morrow} proved that these three definitions are equivalent in the cprn case. In this paper, we investigate to what extent this equivalence holds in the rnpd case.

\section{Local Moves}
There are five local transformations on cprns that do not change the response matrix \cite{deverdiere_gitler_vertigan}.  We list these below.  Note that the networks we draw below are components of a possibly larger network.  The grey vertices indicate where the components in the pictures may be connected to the rest of the graph.  That is, they are vertices that may be either boundary vertices or internal vertices and may have other edges incident to them.  In the cases where the arrow goes only in one direction, we may still apply the transformation in the opposite direction, but there is not a unique way to do so.
\begin{itemize}
    \item Loop removal
    
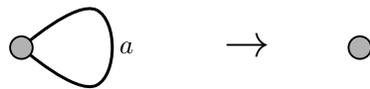
\begin{center}
    \begin{tikzpicture}
    \node at (1.4,0) {$a$};
    \draw [line width=0.4mm,black] plot [smooth, tension=1] coordinates { (0,0) (1,0.5) (1,-0.5) (0,0) };
    \draw [line width=0.25mm, fill=gray!60] (0,0) circle (1.5mm);
    \node at (3,0) {\LARGE$\rightarrow$};
    \draw [line width=0.25mm, fill=gray!60] (4.5,0) circle (1.5mm);
    \fill[white] (5,0) rectangle (6,0);
    \end{tikzpicture}
    \end{center}
    \item Pendant removal
    \begin{center}
    \begin{tikzpicture}
    \path[-, line width=0.4mm]
    (2,0) edge node[below]{$a$} (0,0);
    \draw [line width=0.25mm, fill=gray!60] (0,0) circle (1.5mm);
    \draw [line width=0.25mm, fill=black] (2,0) circle (1.5mm);
    \node at (3.5,0) {\LARGE$\rightarrow$};
    \draw [line width=0.25mm, fill=gray!60] (5,0) circle (1.5mm);
    \fill[white] (6,0) rectangle (7,0);
    \end{tikzpicture}
    \end{center}
    \item Series transformation
    
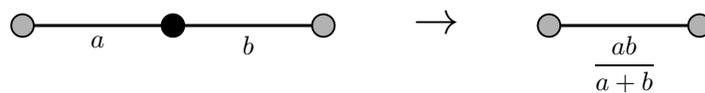
\begin{center}
    \begin{tikzpicture}
    \path[-, line width=0.4mm]
    (2,0) edge node[below]{$a$} (0,0)
    (2,0) edge node[below]{$b$} (4,0);
    \draw [line width=0.25mm, fill=gray!60] (0,0) circle (1.5mm);
    \draw [line width=0.25mm, fill=black] (2,0) circle (1.5mm);
    \draw [line width=0.25mm, fill=gray!60] (4,0) circle (1.5mm);
    \node at (5.5,0) {\LARGE$\rightarrow$};
    \path[-, line width=0.4mm]
    (7,0) edge node[below]{$\displaystyle\frac{ab}{a+b}$} (9,0);
    \draw [line width=0.25mm, fill=gray!60] (7,0) circle (1.5mm);
    \draw [line width=0.25mm, fill=gray!60] (9,0) circle (1.5mm);
    \fill[white] (10,0) rectangle (11,0);
    \end{tikzpicture}
    \end{center}
    \item Parallel transformation
    
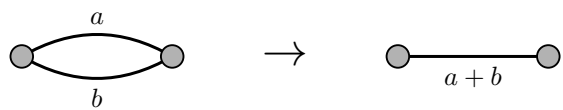
\begin{center}
    \begin{tikzpicture}
    \path[-, line width=0.4mm]
    (2,0) edge[bend right] node[above]{$a$} (0,0)
    (2,0) edge[bend left] node[below]{$b$} (0,0);
    \draw [line width=0.25mm, fill=gray!60] (0,0) circle (1.5mm);
    \draw [line width=0.25mm, fill=gray!60] (2,0) circle (1.5mm);
    \node at (3.5,0) {\LARGE$\rightarrow$};
    \path[-, line width=0.4mm]
    (7,0) edge node[below]{$a+b$} (5,0);
    \draw [line width=0.25mm, fill=gray!60] (7,0) circle (1.5mm);
    \draw [line width=0.25mm, fill=gray!60] (5,0) circle (1.5mm);
    \end{tikzpicture}
    \end{center}
    
    \newpage
    \item Y-$\Delta$ move
    
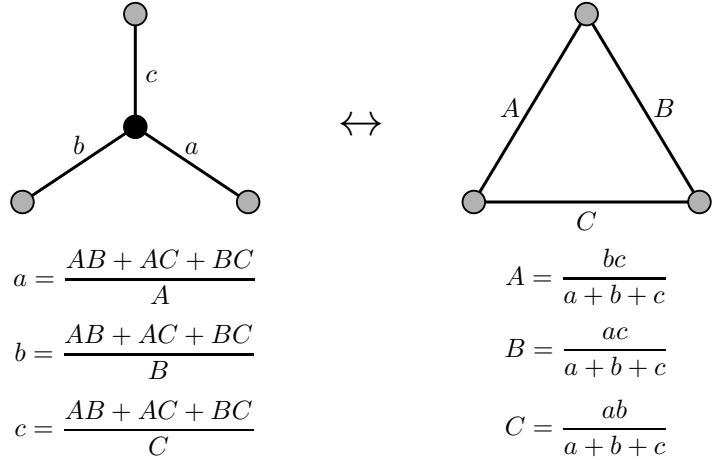
\begin{center}
    \begin{tikzpicture}
    \path[-, line width=0.4mm]
    (0,0) edge node[above]{$b$} (1.5,1)
    (1.5,2.35) edge node[right]{$c$} (1.5,1)
    (3,0) edge node[above]{$a$} (1.5,1);
    \draw [line width=0.25mm, fill=gray!60] (0,0) circle (1.5mm);
    \draw [line width=0.25mm, fill=gray!60] (1.5,2.5) circle (1.5mm);
    \draw [line width=0.25mm, fill=gray!60] (3,0) circle (1.5mm);
    \draw [line width=0.25mm, fill=black] (1.5,1) circle (1.5mm);
    \node at (4.5,1) {\LARGE$\leftrightarrow$};
    \path[-, line width=0.4mm]
    (6,0) edge node[left]{$A$} (7.5,2.5)
    (7.5,2.5) edge node[right]{$B$} (9,0)
    (9,0) edge node[below]{$C$} (6,0);
    \draw [line width=0.25mm, fill=gray!60] (6,0) circle (1.5mm);
    \draw [line width=0.25mm, fill=gray!60] (7.5,2.5) circle (1.5mm);
    \draw [line width=0.25mm, fill=gray!60] (9,0) circle (1.5mm);
    \node at (1.5,-1) {$\displaystyle a=\frac{AB+AC+BC}{A}$};
    \node at (1.5,-2) {$\displaystyle b=\frac{AB+AC+BC}{B}$};
    \node at (1.5,-3) {$\displaystyle c=\frac{AB+AC+BC}{C}$};
    \node at (7.5,-1) {$\displaystyle A=\frac{bc}{a+b+c}$};
    \node at (7.5,-2) {$\displaystyle B=\frac{ac}{a+b+c}$};
    \node at (7.5,-3) {$\displaystyle C=\frac{ab}{a+b+c}$};
    \end{tikzpicture}
    \end{center}
\end{itemize}

All of these moves are valid for rnpds. We also have two new local moves for rnpds that do not exist for cprns:

\begin{itemize}
    \item Antenna Jumping
    
    We define an \textit{antenna} to be an interior boundary vertex of degree one. If we have such a vertex and it is adjascent to vertex $v$, then \textit{antenna jumping} is performed by jumping the antenna over an edge incident to $v$ and into a different face. The conductances are unchanged by this move.
    
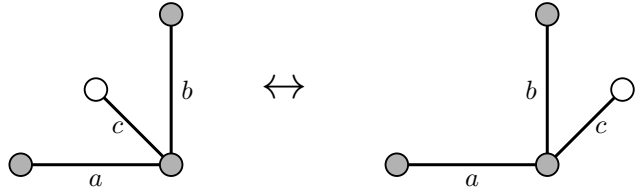
\begin{center}
    \begin{tikzpicture}
    \path[-, line width=0.4mm]
    (2,0) edge node[below]{$a$} (0,0)
    (2,0) edge node[right]{$b$} (2,2)
    (2,0) edge node[left]{$c$} (1,1);
    \fill[white] (-1,0) rectangle (0,1);
    \draw [line width=0.25mm, fill=white] (1,1) circle (1.5mm);
    \draw [line width=0.25mm, fill=gray!60] (0,0) circle (1.5mm);
    \draw [line width=0.25mm, fill=gray!60] (2,0) circle (1.5mm);
    \draw [line width=0.25mm, fill=gray!60] (2,2) circle (1.5mm);
    \node at (3.5,1) {\LARGE$\leftrightarrow$};
    \path[-, line width=0.4mm]
    (7,0) edge node[below]{$a$} (5,0)
    (7,0) edge node[left]{$b$} (7,2)
    (7,0) edge node[right]{$c$} (8,1);
    \draw [line width=0.25mm, fill=white] (8,1) circle (1.5mm);
    \draw [line width=0.25mm, fill=gray!60] (5,0) circle (1.5mm);
    \draw [line width=0.25mm, fill=gray!60] (7,0) circle (1.5mm);
    \draw [line width=0.25mm, fill=gray!60] (7,2) circle (1.5mm);
    \end{tikzpicture}
    \end{center}
    \item Antenna Absorption
    
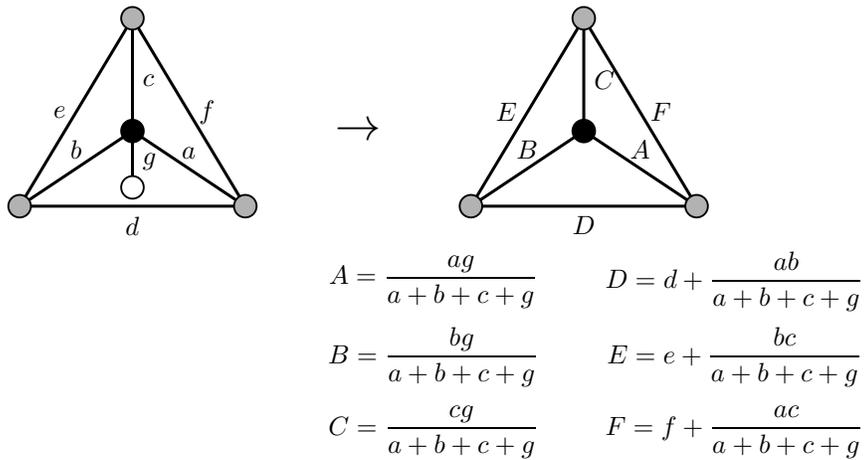
\begin{center}
    \begin{tikzpicture}
    \path[-, line width=0.4mm]
    (0,0) edge node[above]{$b$} (1.5,1)
    (1.5,2.35) edge node[right]{$c$} (1.5,1)
    (3,0) edge node[above]{$a$} (1.5,1)
    (0,0) edge node[left]{$e$} (1.5,2.5)
    (1.5,2.5) edge node[right]{$f$} (3,0)
    (3,0) edge node[below]{$d$} (0,0)
    (1.5,0.2) edge node[right]{$g$} (1.5,1);
    \draw [line width=0.25mm, fill=gray!60] (0,0) circle (1.5mm);
    \draw [line width=0.25mm, fill=gray!60] (1.5,2.5) circle (1.5mm);
    \draw [line width=0.25mm, fill=gray!60] (3,0) circle (1.5mm);
    \draw [line width=0.25mm, fill=black] (1.5,1) circle (1.5mm);
    \draw [line width=0.25mm, fill=white] (1.5,0.25) circle (1.5mm);
    \node at (4.5,1) {\LARGE$\rightarrow$};
    \path[-, line width=0.4mm]
    (6,0) edge node[above]{$B$} (7.5,1)
    (7.5,2.35) edge node[right]{$C$} (7.5,1)
    (9,0) edge node[above]{$A$} (7.5,1)
    (6,0) edge node[left]{$E$} (7.5,2.5)
    (7.5,2.5) edge node[right]{$F$} (9,0)
    (9,0) edge node[below]{$D$} (6,0);
    \draw [line width=0.25mm, fill=gray!60] (6,0) circle (1.5mm);
    \draw [line width=0.25mm, fill=gray!60] (7.5,2.5) circle (1.5mm);
    \draw [line width=0.25mm, fill=gray!60] (9,0) circle (1.5mm);
    \draw [line width=0.25mm, fill=black] (7.5,1) circle (1.5mm);
    \node at (5.5,-1) {$\displaystyle A=\frac{ag}{a+b+c+g}$};
    \node at (5.5,-2) {$\displaystyle B=\frac{bg}{a+b+c+g}$};
    \node at (5.5,-3) {$\displaystyle C=\frac{cg}{a+b+c+g}$};
    \node at (9.5,-1) {$\displaystyle D=d+\frac{ab}{a+b+c+g}$};
    \node at (9.5,-2) {$\displaystyle E=e+\frac{bc}{a+b+c+g}$};
    \node at (9.5,-3) {$\displaystyle F=f+\frac{ac}{a+b+c+g}$};
    \node at (-0.5,-3) {\phantom{$\displaystyle F=f+\frac{ac}{a+b+c+g}$}};
    \end{tikzpicture}
    \end{center}
\end{itemize}

\begin{conj}
Any two electrically equivalent graphs associated to rnpds are connected by the local moves stated in this section.
\end{conj}

\section{Medial Graphs and Z-sequences}
%This section has two things that deserve to be called theorems IMO - all but the last one could/should be combined into one.

\begin{defn}[\cite{curtis_ingerman_morrow}]
For an cprn $\Gamma$ with $n$ boundary vertices, $v_1,...,v_n$ clockwise, place $2n$ points, $t_1,...,t_{2n}$, around the boundary of the disk in clockwise order such that for even $i$, $t_{i-1}$ and $t_i$ are on either side of $v_{i/2}$.  That is, between $v_j$ and $v_{j+1}$ we have, in clockwise order, $v_j, t_{2j}, t_{2j+1}, v_{j+1}$ where indices are modulo $n$ for the $v_i$'s and modulo $2n$ for the $t_i$'s.  For each edge $e$ in $\Gamma$, let $m_e$ be its midpoint.
Then the \textit{medial graph}, $M(\Gamma)$, of $\Gamma$ is the graph with vertices $\{t_i\ |\ 1\leq i\leq 2n\} \cup \{m_e\ |\ e\in E\}$ and edges $\{(m_e,m_f)\ |\ e,f\text{ edges in }\Gamma\text{ on the same face and sharing a vertex}\}\cup\{(m_e,m_e)\ |\ m_e\text{ a spike into a face of }\Gamma\}\cup\{(t_i,m_e)\ |\ t_i,e\text{ on the same face, }i\text{ odd, }e\text{ an edge connected}\\\text{to }v_{(i+1)/2} \}\cup\{(t_i,m_e)\ |\ t_i,e\text{ on the same face, }i\text{ even, }e\text{ an edge connected to }v_{i/2} \}$.
\end{defn}

\begin{ex}\label{ex:medial-graph}
The medial graph of the cprn in Figure~\ref{fig:cprn-rnpd} is drawn below.
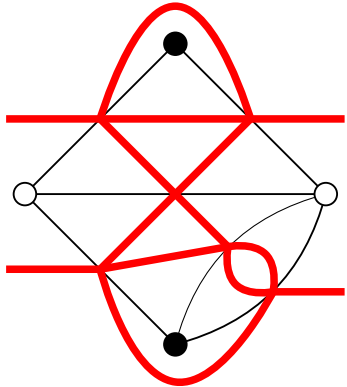
\begin{center}
    \begin{tikzpicture}[scale=1]
          %Draw Graph
          \draw [line width = 0.25mm, fill = black] (0, 2)edge (2,0);
          \draw [line width = 0.25mm, fill = black] (0, 2)edge (-2,0);
          \draw [line width = 0.25mm, fill = black] (0, -2)edge[bend right] (2,0);
          \draw [line width = 0.25mm, fill = black] (0, -2)edge (-2,0);
          \draw [line width = 0.25mm, fill = black]  (2,0)edge (-2,0) ;
          \draw (2,0) edge[bend right] (0,-2);
          
          \draw [line width=0.25mm, fill=black] (0, 2) circle (1.5mm);
          \draw [line width=0.25mm, fill=black] (0, -2) circle (1.5mm);
          \draw [line width=0.25mm, fill=white] (-2,0) circle (1.5mm);
          \draw [line width=0.25mm, fill=white] (2,0) circle (1.5mm);
          
          %Draw Medial Strands
          \draw [line width=1mm, red] plot [smooth, tension=0] coordinates{(-2.25, 1) (-2,1) (-1, 1) (1, 1) (2, 1) (2.25, 1)};

          \draw [line width=1mm, red] plot [smooth, tension=0] coordinates{(-1, -1) (0, 0) (1, 1)};
          \draw [line width=1mm, red] plot [smooth, tension=1] coordinates{(1, 1) (0, 2.5) (-1, 1)};
          \draw [line width=1mm, red] plot [smooth, tension=0] coordinates{(-1, 1) (0, 0) (.7, -.7)};
          \draw [line width=1mm, red] plot [smooth, tension=1] coordinates{(.7, -.7) (.8, -1.2) (1.3, -1.3)};
          \draw [line width=1mm, red] plot [smooth, tension=1] coordinates{(1.3, -1.3) (2, -1.3) (2.25, -1.3)};
          
          \draw [line width=1mm, red] plot [smooth, tension=0] coordinates{(-2.25, -1) (-2, -1) (-1, -1)  (.7, -.7)};
          \draw [line width=1mm, red] plot [smooth, tension=1] coordinates{(.7, -.7) (1.2, -.8) (1.3, -1.3)};
          \draw [line width=1mm, red] plot [smooth, tension=0] coordinates{(1.3, -1.3)};
          \draw [line width=1mm, red] plot [smooth, tension=1] coordinates{(1.3, -1.3) (0, -2.5) (-1, -1)};
          \draw [line width=1mm, red] plot [smooth, tension=0] coordinates{(-1, -1)};
    \end{tikzpicture}
\end{center}
\end{ex}

\begin{rem}
For an edge $e$ of cprn $\Gamma$, $m_e$ has degree $4$ in $M(\Gamma)$. This means that if we approach $m_e$ via some edge in $M(\Gamma)$, we can turn right, turn left, or go straight.
\end{rem}

\begin{defn}
We define an equivalence relation $\equiv$ on the edges of $M(\Gamma)$ as follows.  For every $e$ an edge of $\Gamma$ with corresponding vertex $m_e$ in $M(\Gamma)$, let the four medial edges adjacent to $m_e$ be $v_{e,1}, v_{e,2}, v_{e,3},$ and $v_{e,4}$ in clockwise order. Then $v_{e,1} \equiv v_{e,3}$ and $v_{e,2} \equiv v_{e,4}$. Then the \textit{medial strands} of $M(\Gamma)$ are the equivalence classes of edges of $M(\Gamma)$ induced by $\equiv$.
In other words, a medial strand is a path in $M(\Gamma)$ where we always go straight through a vertex $m_e$.  We frequently smooth the corners of a strand in depictions of the medial graph.
\end{defn}

\begin{ex}
The medial graph depicted in Example~\ref{ex:medial-graph} has two medial strands, as shown below.
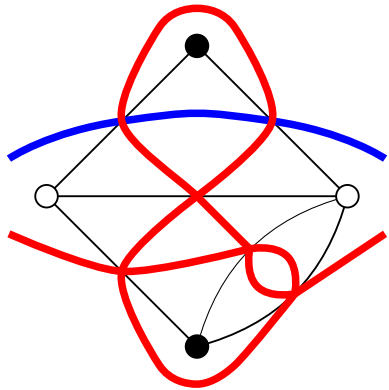
\begin{center}
    \begin{tikzpicture}[scale=1]
          %Draw Graph
          \draw [line width = 0.25mm, fill = black] (0, 2)edge (2,0);
          \draw [line width = 0.25mm, fill = black] (0, 2)edge (-2,0);
          \draw [line width = 0.25mm, fill = black] (0, -2)edge[bend right] (2,0);
          \draw [line width = 0.25mm, fill = black] (0, -2)edge (-2,0);
          \draw [line width = 0.25mm, fill = black]  (2,0)edge (-2,0) ;
          \draw (2,0) edge[bend right] (0,-2);
          
          \draw [line width=0.25mm, fill=black] (0, 2) circle (1.5mm);
          \draw [line width=0.25mm, fill=black] (0, -2) circle (1.5mm);
          \draw [line width=0.25mm, fill=white] (-2,0) circle (1.5mm);
          \draw [line width=0.25mm, fill=white] (2,0) circle (1.5mm);
          
          %Draw Medial Strands
          \draw [line width=1mm, blue] plot [smooth, tension=1] coordinates{(-2.5,.5) (-1, 1) (1, 1) (2.5, .5)};
          \draw [line width=1mm, red] plot [smooth, tension=.5] coordinates{(0, 0) (1, 1) (.5, 2.25) (0, 2.5) (-.5, 2.25) (-1, 1) (0, 0) (.7, -.7)};
          \draw [line width=1mm, red] plot [smooth, tension=1] coordinates{(.7, -.7) (1.2, -.8) (1.3, -1.3)};
          \draw [line width=1mm, red] plot [smooth, tension=1] coordinates{(1.3, -1.3) (2.5, -.5)};
          \draw [line width=1mm, red] plot [smooth, tension=.5] coordinates{(-2.5, -.5) (-1, -1)  (.7, -.7)};
         \draw [line width=1mm, red] plot [smooth, tension=1] coordinates{(.7, -.7) (.8, -1.2) (1.3, -1.3)};
        \draw [line width=1mm, red] plot [smooth, tension=.5] coordinates{(1.3, -1.3) (.5, -2.25) (0, -2.5) (-.5, -2.25) (-1, -1) (0, 0)};
    \end{tikzpicture}
\end{center}
\end{ex}

\begin{rem}
Each vertex $t_i$ has degree $1$ in $M(\Gamma)$.  This means that each $t_i$ is an endpoint of exactly one medial strand.
\end{rem}

\begin{defn}\cite{curtis_ingerman_morrow}
Let $\Gamma$ be a cprn with $n$ boundary vertices, $v_1,...,v_n$ clockwise, and let $t_1,...,t_{2n}$ be the endpoints of medial strands. Label the strand beginning at $t_1$ as with a 1.  The remaining strands are labeled 2 through $n$ so that if $i<j$, the endpoints of the strand $i$ are $t_a$ and $t_b$, and the endpoints of strand $j$ are $t_c$ and $t_d$, we have either $a<c,d$ or $b<c,d$.  In other words, the first endpoint of strand $i$ appears clockwise from $t_1$ before the first endpoint of $j$.  For each $i\in\{1, 2, \cdots, 2n\}$ let $z_i$ be the number associated with the strand with an endpoint at $t_i$. The sequence $z_1,...,z_{2n}$ is called the \textit{$z$-sequence} of $\Gamma$.
\end{defn}

We can now make some analogous definitions for rnpds.

\begin{defn}
For an rnpd $\Gamma$ we can define the \textit{medial graph} $M(\Gamma)$ of $\Gamma$ to be as for a cprn, treating the interior boundary vertex as an internal vertex.  We obtain \emph{medial strands} as in the cprn case.
\end{defn}

\begin{ex}\label{ex:rnpd-medial-graph}
The medial graph of the rnpd in Figure \ref{fig:cprn-rnpd} is 
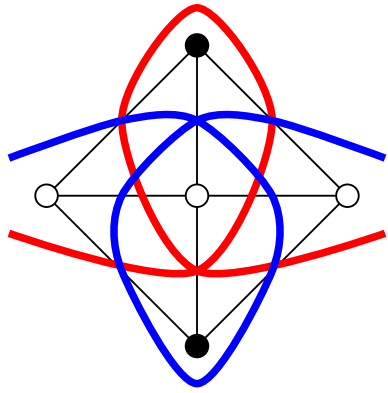
\begin{center}
    \begin{tikzpicture}[scale=1]
          %Draw Graph
          \draw [line width = 0.25mm, fill = black] (0, 2)edge (2,0);
          \draw [line width = 0.25mm, fill = black] (0, 2)edge (-2,0);
          \draw [line width = 0.25mm, fill = black] (0, -2)edge (2,0);
          \draw [line width = 0.25mm, fill = black] (0, -2)edge (-2,0);
          \draw [line width = 0.25mm, fill = black]  (2,0)edge (0,0) ;
          \draw [line width = 0.25mm, fill = black]  (-2,0)edge (0,0) ;
          \draw [line width = 0.25mm, fill = black]  (0,0)edge (0,-2) ;
          \draw [line width = 0.25mm, fill = black]  (0,0)edge (0,2) ;

          \draw [line width=0.25mm, fill=black] (0, 2) circle (1.5mm);
          \draw [line width=0.25mm, fill=black] (0, -2) circle (1.5mm);
          \draw [line width=0.25mm, fill=white] (-2,0) circle (1.5mm);
          \draw [line width=0.25mm, fill=white] (2,0) circle (1.5mm);
          \draw [line width=0.25mm, fill=white] (0,0) circle (1.5mm);
          
          %Draw Medial Strands
          \draw [line width=1mm, red] plot [smooth, tension=.5] coordinates{(-2.5, -.5) (0, -1) (1, 1) (0, 2.5) (-1, 1) (0, -1)  (2.5, -.5)};
          \draw [line width=1mm, blue] plot [smooth, tension=.5] coordinates{(-2.5, .5) (-1, 1) (0, 1) (1, 0) (1, -1) (0, -2.5) (-1, -1) (-1, 0) (0, 1) (1, 1) (2.5, .5)};
    \end{tikzpicture}
\end{center}
\end{ex}

\begin{defn}
For an rnpd $\Gamma$, we can define the \textit{z-sequence} by beginning with the construction for cprns to get some permutation of the multiset $\{1,1,2,2,...,n,n\}$.  For each $i$, label one strand endpoint of $s_i$ as $s_i^+$ and the other $s_i^-$ such that the strand from $s_i^-$ to $s_i^+$ moves clockwise around $b$. Additionally, if a strand contains a self-intersection, put a bar over its endpoints' labels.  The $z$-sequence of $\Gamma$ is the sequence of labels of strand endpoints, starting at $t_1$ going clockwise.  
\end{defn}

\begin{ex}\label{ex:rnpd-z-seq}
Consider the rnpd from Example~\ref{ex:rnpd-medial-graph}.  If the boundary vertex on the left is $v_1$ and the boundary vertex on the right is $v_2$, then the $z$-sequence for this rnpd is $\overline{1^+},\overline{2^-},\overline{2^+},\overline{1^-}$.
\end{ex}

\begin{defn}
A \textit{motion} on a triangular face $ABC$ of a medial graph is defined by the following (local) transformation:

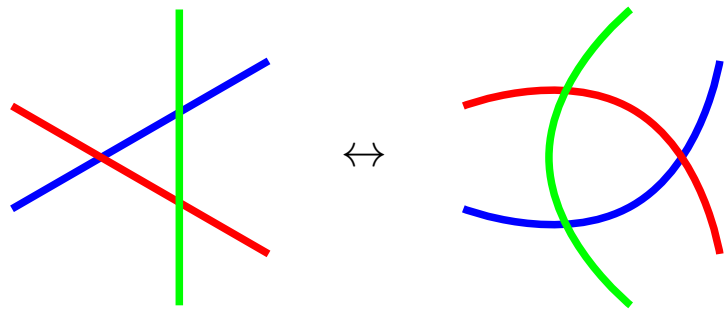
\begin{center}
    \begin{tikzpicture}[scale=2]

          %Draw left medial strands
          \draw [line width=1mm, blue] plot [smooth, tension=1] coordinates{(.766, 0.642)  (-0.939, -0.342)};
          
          \draw [line width=1mm, red] plot [smooth, tension=1] coordinates{(.766, -0.642)  (-0.939, 0.342)};
          
          \draw [line width=1mm, green] plot [smooth, tension=1] coordinates{(0.173, -0.984)  (0.173, 0.984)};
          
          %Draw right medial strands
          \draw [line width=1mm, blue] plot [smooth, tension=1] coordinates{(3.766, 0.642) (3.185, -.320) (3-0.939, -0.342)};
          
          \draw [line width=1mm, red] plot [smooth, tension=1] coordinates{(3.766, -0.642) (3.185, .320) (3-0.939, 0.342)};
          
           \draw [line width=1mm, green] plot [smooth, tension=1] coordinates{(3.173, -0.984) (3-.370, 0) (3.173, 0.984)};
           
          %Draw arrows
          \node at (1.4,0){\LARGE$\leftrightarrow$};
    \end{tikzpicture}
\end{center}

\end{defn}

\begin{rem}
Medial graph motions correspond directly to $Y-\Delta$ moves, so the graphs of two rnpds are $Y-\Delta$ equivalent if and only if their medial graphs are equivalent by motions.
\end{rem}

\begin{defn}
Suppose $s'$ and $t'$ are segments of two medial strands $s$ and $t$ such that $s'$ and $t'$ both have endpoints $a$ and $b$, $s'$ and $t'$ do not intersect themselves or each other between $a$ and $b$, and the medial edges in $s$ and $t$ adjacent to $a$ and $b$ but not on $s'$ and $t'$ are outside the region enclosed by $s'$ and $t'$.
Define a \textit{medial lens} to be the region enclosed by $s$ and $t$ between $a$ and $b$.
\end{defn}

\begin{defn}
Define a \textit{medial loop} to be the region enclosed by a medial strand segment whose endpoints are the same.
\end{defn}

\begin{defn}
Define a \textit{medial circle} to be a medial strand that does not have endpoints.
\end{defn}

\begin{ex}
In the following example medial graph, the green strand is a medial circle. It bounds two lenses with the blue strand. The red strand bounds a medial loop. The blue and red strands bound a medial lens.
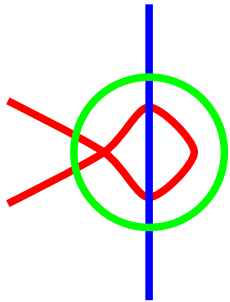
\begin{center}
    \begin{tikzpicture}[scale=2]

          %Draw left medial strands

          \draw [line width=1mm, red] plot [smooth, tension=0.5] coordinates{(-0.939, 0.342) (-0.3, 0) (0, -0.3) (0.3,0) (0, 0.3) (-0.3, 0) (-0.939, -0.342)};
          
          \draw [line width=1mm, blue] plot [smooth, tension=1] coordinates{(0, -0.984)  (0, 0.984)};
          \draw [line width=1mm, green] (0,0) circle (0.5);
    \end{tikzpicture}
\end{center}
\end{ex}

         %\com{Something about how cprns are irreducible iff the medial graph has none of these things.  Not true in the rnpd case.  However, we are able to prove that if an rnpd is irreducible *summary of results*.}

With these definitions in hand, we may state the first of two main theorems of this section.

\begin{thm}
\label{FirstMedialTheorem}
An rnpd is irreducible by $Y-\Delta$ moves, series reductions, parallel reductions, pendant removal, self-edge removal, and antenna jumping if and only if  the following conditions all hold:
\begin{enumerate}[(a)]
    \item $M(\Gamma)$ contains no medial circles.
    \item $M(\Gamma)$ contains at most one self-intersecting strand. Furthermore, if such a strand exists, it intersects itself only once, with the loop it creates containing the interior boundary vertex.
    
    \item Any two distinct medial strands of $M(\Gamma)$ intersect at most twice. Furthermore, if two distinct medial strands intersect twice, creating a lens, the lens they create contains the interior boundary vertex.
\end{enumerate}
\end{thm}

We first prove the backwards direction.

\begin{proof}

Edges in series and parallel induce a medial lens face in the medial graph. Pendants and self-edges induce a medial loop face in the medial graph. Thus, no graph satisfying (b) and (c) can be reduced only by series reductions, parallel reductions, self-edge removal, and pendant removal.\\
Medial circles remain medial circles under medial graph motions and antenna jumping. Similarly, medial lenses and medial loops not containing $b$ remain medial lenses and medial loops not containing $b$ under motions and antenna jumping.
Thus, if a graph satisfies properties (a), (b), and (c), it will always satisfy the properties after $Y-\Delta$ moves and antenna jumping, and hence will never be able to be reduced by series reductions, parallel  reductions,  pendant removal, or self-edge  removal.
\end{proof}

To prove the forwards direction, we require several lemmas.

\begin{lemma}
\label{NoEmptyLenses}
In the medial graph of an  rnpd irreducible by $Y-\Delta$ moves, series reductions, parallel reductions, pendant removal, and self-edge removal:
\begin{enumerate}[(a)]
    \item Any medial lens contains the interior boundary vertex.
    \item Any medial loop contains the interior boundary vertex.
    \item Any medial circle contains the interior boundary vertex.
\end{enumerate}
\begin{proof}
Suppose our medial graph has a medial lens.  We follow the algorithm in Lemma 6.2 of \cite{curtis_ingerman_morrow}, noting that we may still perform motions on any triangular face in the medial graph that does not contain the boundary vertex.  For a lens that does not contain the boundary vertex, none of its faces contain the boundary vertex. Following the algorithm in Lemma 6.2 will never cause the lens to contain new vertices.  Thus, for rnpds, any medial graph with a medial lens not containing the boundary vertex is equivalent by motions to one with a lens (still not containing the boundary vertex) as a face.

We can apply the same argument for medial loops, in which case we obtain a medial loop face, and medial circles, in which case we obtain a medial circle face.

Lens faces may be reduced by a series or parallel transformation (depending on whether a vertex is inside the lens) and medial loop faces may be reduced by pendant removal or loop removal (depending on whether a vertex is inside the loop).  In both cases, we have a contradiction because our graph was reduced. Medial circle faces cannot exist.  
\end{proof}
\end{lemma}

\begin{lemma}
\label{CanMakeTriangleEmpty}
Let $p,q,r$ be pairwise-intersecting medial strands in an rnpd $\Gamma$ irreducible by $Y-\Delta$ moves, series reductions, parallel reductions, pendant removal, and self-edge removal.  Define $v_{qr}$ be an intersection of strands $q$ and $r$ and define $v_{pr}$, and $v_{pq}$ similarly. 
Suppose that the segments of strands $p,q$, and $r$ between $v_{pr}$ and $v_{pq}$, $v_{pq}$ and $v_{qr}$, and $v_{pr}$ and $v_{qr}$ form the sides of a triangle $T$.  That is, these strand segments do not intersect themselves or each other except for at $v_{qr},v_{pr}$, and $v_{pq}$.
If $T$ does not contain $b$, then we can apply a sequence of motions on faces of $T$ to obtain a network where $T$ is a face of the medial graph.

\begin{proof}
We prove this inductively on the number of faces inside $T$. If $T$ only has one face, no motions are required.

If $T$ has more than one face, some strand $s$ intersects the boundary of $T$ twice.  If these intersections occur on the same strand, there would be a lens inside $T$.  Since $T$ does not contain $b$, this contradicts Lemma~\ref{NoEmptyLenses}.  So, without loss of generality, we can assume these intersections occur on strands $q$ and $r$ and that the intersection points are $v_{qs}$ and $v_{rs}$. Then $q,r$, and $s$ enclose a region $U$ with vertices $v_{qr}$, $v_{qs}$, and $v_{rs}$. Since $U$ is strictly inside $T$, $U$ has fewer faces than $T$. Thus, by induction, we can perform motions so that there are no intersections inside $U$.  These motions will not increase the number of faces in $T$. Next we can perform a motion to move $U$ outside $T$, decreasing the number of faces in $T$ by one.  By induction, we can perform another sequence of motions to make $T$ a face of the medial graph.
\end{proof}
\end{lemma}

\begin{lemma}
\label{OneMedialLoop}
In an rnpd irreducible by $Y-\Delta$ moves, series reductions, parallel reductions, pendant removal, self-edge removal, and antenna jumping with no medial circles, there exists at most one medial loop. That is, there exists at most once medial strand with a self-intersection, and if such a strand exists, it intersects itself only once.
\begin{proof}

For the sake of contradiction, suppose we have at least two medial loops. Let $\ell$ be a medial loop contains no other medial loops within its interior.  By the previous lemma, the interior boundary vertex $b$ must be inside $\ell$. 

Claim 1: We can use motions to make $\ell$ a face in the medial graph. 

Any strand segment $s$ in $\ell$ with boundary on $\ell$ contains no self-intersection. Then it divides $\ell$ into a lens and a triangle $T$. The lens must then contain the interior boundary vertex $b$, so $T$ does not. By Lemma~\ref{CanMakeTriangleEmpty} we can use motions to make $T$ empty.  These motions do not increase the number of strand segments in $\ell$.  After we make $T$ empty, a motion on $T$ decreases the number of strand segments in $\ell$. Applying this process repeatedly, we can make $\ell$ a face, proving Claim 1.

By Claim 1, up to electrical equivalence, we may assume $\ell$ is a face in the medial graph. This face must correspond to the vertex $b$ of the rnpd by part (b) of Lemma~\ref{NoEmptyLenses}.  Since every medial loop contains $b$, every medial loop contains $\ell$. Without loss of generality, let $\ell'$ be such that $\ell$ is the only loop contained in $\ell'$.

Claim 2: We can use motions to make $\ell'$ contain only $\ell$.

Let the self-intersection of $\ell'$ be $a$. As in the proof of Claim 1, any strand segment $s$ in $\ell'$ with endpoints $s_1$ and $s_2$ on $\ell'$ contains no self-intersection. Since $s$ cannot intersect $\ell$ or create a lens, $s_1$, $s_2$, and $a$ create a triangle $T$ that does not contain $b$. 

By Lemma~\ref{CanMakeTriangleEmpty} we can use motions to make $T$ empty.  These motions do not increase the number of strand segments in $\ell'$.  After we make $T$ empty, a motion on $T$ decreases the number of strand segments in $\ell'$. Applying this process repeatedly, we can make $\ell'$ contain only $\ell$, proving Claim 2.

Now we're in one of three situations, shown below. 

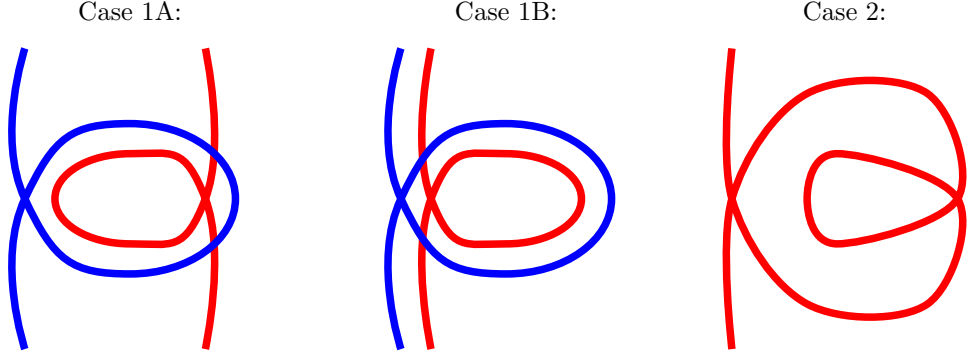
\begin{center}
    \begin{tikzpicture}[scale=2]

          %Case 1A
          \draw (-0.5,1.25) node {Case 1A:};
          \draw [line width=1mm, red] plot [smooth, tension=1] coordinates{(0, 1) (0, 0) (-0.5,-.3) (-1, 0) (-0.5,.3) (0, 0) (0, -1)};
          \draw [line width=1mm, blue] plot [smooth, tension=1] coordinates{(-1.2, 1) (-1.2, 0) (-0.5,-.5) (0.2, 0) (-0.5,.5) (-1.2, 0) (-1.2, -1)};
          
          %Case 1B
          \draw (2,1.25) node {Case 1B:};
          \draw [line width=1mm, red] plot [smooth, tension=1] coordinates{(1.5, 1) (1.5, 0) (2,-.3) (2.5, 0) (2,.3) (1.5, 0) (1.5, -1)};
          \draw [line width=1mm, blue] plot [smooth, tension=1] coordinates{(1.3, 1) (1.3, 0) (2,-.5) (2.7, 0) (2,.5) (1.3, 0) (1.3, -1)};
          
          %Case 2
          \draw (4.25,1.25) node {Case 2:};
          \draw [line width=1mm, red] plot [smooth, tension=0.6] coordinates{(3.5, 1) (3.5, 0) (4,-0.7) (4.8,-0.7) (5,0) (4.2,0.3) (4, 0) (4.2,-0.3) (5,0) (4.8,0.7) (4,0.7) (3.5, 0) (3.5, -1)};

    \end{tikzpicture}
\end{center}

In cases 1A and 1B, antenna jumping yields a medial lens not containing $b$ (note the parallel edges in 1A and the parallel edges after a $Y-\Delta$ move in 1B). In case $2$, antenna jumping yields a medial loop not containing $b$ (note the self-edge).  This is a contradiction with Lemma~\ref{NoEmptyLenses}, so we have proven the lemma.
\end{proof}
\end{lemma}

To prove the next Lemma, we need to introduce some notation.

\begin{defn}
Let $\Gamma$ be a graph with vertex set $V$. Let $V' \subseteq V,\ V''$ be all vertices in $V'$ or adjacent to $V',\ E'$ be the edges of $\Gamma$ with both endpoints in $V'$, and $E''$ be the edges of $\Gamma$ with at least one endpoint in $V'$. Then the \textit{strong restriction} of $\Gamma$ to $V'$ is $\Gamma'=(V',E')$ and the \textit{weak restriction} of $\Gamma$ to $V'$ is $\Gamma''=(V'',E'')$. The boundary vertices of $\Gamma'$ are $(B\cap V')\cup\{v\in V'\ |\ v\text{ is adjacent to a vertex in }V\setminus V'\}$. The boundary vertices of $\Gamma''$ are $(B\cap V'')\cup(V''\setminus V')$.
\end{defn}

\begin{rem}
Both the strong restriction and weak restrictions of $\Gamma$ to $V'$ are \textit{subgraphs} of $\Gamma$.
\end{rem}

\begin{rem}
All graphs $\Gamma$ are equal to the strong restriction to their vertex set. Assuming connectedness, they are also equal to the weak restriction to their interior vertex set.
\end{rem}

\begin{lemma}
\label{NoMedialCircles}
The medial graph of an irreducible rnpd does not contain any medial circles. 
\begin{proof}
For the sake of contradiction, suppose there exists an irreducible rnpd $\Gamma$ whose medial graph $M(\Gamma)$ contains a medial circle $c$. Without loss of generality, assume $c$ contains no medial circles in its interior.
Let $\Gamma'$ be the strong restriction of $\Gamma$ to the vertices inside $c$. Then $M(\Gamma')$ is the restriction of $M(\Gamma)$ to medial edges strictly contained in $c$. Note that $M(\Gamma')$ contains no medial circles and at most one medial loop.
We consider several disjoint cases:\\

Case 1: The medial circle $c$ contains no self-intersection. 

If a strand segment inside $c$ does not have a self-intersection, then the strand creates a two lenses with $c$.  Only one of those lenses can contain the boundary vertex, so by Lemma 2 we find $\Gamma$ is not reduced.  This is a contradiction, so every strand segment inside $c$ with endpoints on $c$ contains a self-intersection.

\newpage
Since $M(\Gamma')$ contains at most one medial loop, $c$ contains at most one strand segment.  The graph is connected, so $c$ can't be empty. The medial graph containing only $c$ and a medial loop as segments is shown below.

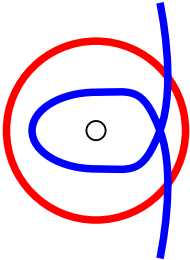
\begin{center}
    \begin{tikzpicture}[scale=0.85]
        \draw [line width=1mm, red] (-5,0) circle (1.4);
        \draw [line width=1mm, blue] plot [smooth, tension=1] coordinates{(-4, 2) (-4, 0) (-5,-.6) (-6, 0) (-5,.6) (-4, 0) (-4, -2)};
        \draw [line width=0.25mm, fill=white] (-5,0) circle (1.5mm);
    \end{tikzpicture}
\end{center}

This is not a possible configuration as its network is reducible.\\

Case 2: The medial circle $c$ contains exactly one self-intersection.

The medial circle $c$ cannot form a figure-eight, as that creates two disjoint loops. Thus it must have the following form:
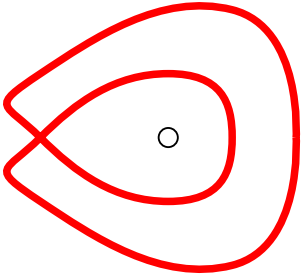
\begin{center}
    \begin{tikzpicture}[scale=0.85]
        \draw [line width=1mm, red] plot [smooth, tension=1] coordinates{(2,0) (1,2) (-2,1) (-2,0) (0,-1) (1,0) (0,1) (-2,0) (-2,-1) (1,-2) (2,0)};
        \draw [line width=0.25mm, fill=white] (0,0) circle (1.5mm);
    \end{tikzpicture}
\end{center}

Just as in case 1, any strand segment inside $c$ without self-intersections creates a lens not containing $b$.
From earlier, $c$ can strictly contain at most one medial loop which can be its only strand. This yields three configurations up to motions for strands in $c$:
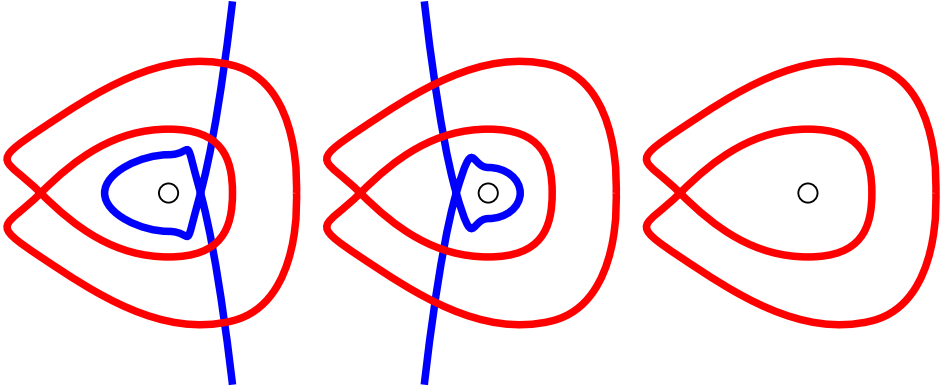
\begin{center}
    \begin{tikzpicture}[scale=0.85]
        %case 1
        \draw [line width=1mm, blue] plot [smooth, tension=1] coordinates{(-4, 3) (-4.5, 0) (-5,-.6) (-6, 0) (-5,.6) (-4.5, 0) (-4, -3)};
        \draw [line width=1mm, red] plot [smooth, tension=1] coordinates{(-3,0) (-4,2) (-7,1) (-7,0) (-5,-1) (-4,0) (-5,1) (-7,0) (-7,-1) (-4,-2) (-3,0)};
        \draw [line width=0.25mm, fill=white] (-5,0) circle (1.5mm);
        
        %case 2
        \draw [line width=1mm, blue] plot [smooth, tension=1] coordinates{(-1, 3) (-0.5, 0) (0,-.4) (0.5, 0) (0,.4) (-0.5, 0) (-1, -3)};
        \draw [line width=1mm, red] plot [smooth, tension=1] coordinates{(2,0) (1,2) (-2,1) (-2,0) (0,-1) (1,0) (0,1) (-2,0) (-2,-1) (1,-2) (2,0)};
        \draw [line width=0.25mm, fill=white] (0,0) circle (1.5mm);
        
        %case 3
        \draw [line width=1mm, red] plot [smooth, tension=1] coordinates{(7,0) (6,2) (3,1) (3,0) (5,-1) (6,0) (5,1) (3,0) (3,-1) (6,-2) (7,0)};
        \draw [line width=0.25mm, fill=white] (5,0) circle (1.5mm);
    \end{tikzpicture}
\end{center}

The first two are reducible. The third can't exist, as the medial graph of a connected graph is connected and $c$ is not the full medial graph of any rnpd.\\

Case 3: The medial circle $c$ contains at least two self-intersections.

Let $p$ be a point on $c$. Then without loss of generality, moving along the to reach the first intersection $e_1$ of the strand segment, we create a counterclockwise loop. This must contain $b$. Then, continuing along $c$ to the second intersection $e_2$ we cannot go clockwise, as that creates a lens not containing $b$. Thus we must go counterclockwise, and a segment of $c$ mus look as below.
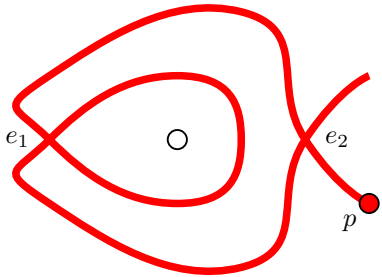
\begin{center}
    \begin{tikzpicture}[scale=0.85]
        \draw [line width=1mm, red] plot [smooth, tension=1] coordinates{(9,-1) (8,0) (7,2) (4,1) (4,0) (6,-1) (7,0) (6,1) (4,0) (4,-1) (7,-2) (8,0) (9,1)};
        \draw (8.7,-1.3) node {$p$};
        \draw (3.5, 0) node {$e_1$};
        \draw (8.5, 0) node {$e_2$};
        \draw [line width=0.25mm, fill=red] (9,-1) circle (1.5mm);
        \draw [line width=0.25mm, fill=white] (6,0) circle (1.5mm);
    \end{tikzpicture}
\end{center}

Now, let $\ell$ be the shown strand segment of $c$. Then any not-self-intersecting strand segment $s$ in $\ell$ intersecting the boundary of $\ell$ at endpoints $s_1$ and $s_2$ either creates a lens with $\ell$ not containing $b$ or a triangle $T$ with $s_1$, $s_2$, and $e_2$. The first case is reducible. In the second we may apply motions on regions in $T$ to make $T$ empty. Then a motion inside $T$ decreases the number of strand segments intersecting $\ell$. Thus, we may assume every strand segment $s$ in $\ell$ contains a self-intersection. 

Since the loop with self-intersecting vertex $e_1$ is strictly contained in $c$, no other loops strictly in $c$ exist. Thus, no strands may intersect $\ell$. But this is then reducible.
\end{proof}
\end{lemma}

\begin{lemma}
\label{NoTwoLenses}
In an rnpd irreducible by $Y-\Delta$ moves, pendant removal, self-edge removal, parallel reductions, series reductions, and antenna jumping, any two medial strands intersect at most twice.

\begin{proof}
Assume for contradiction we have two medial strands $p$ and $q$ that intersect (at least) three times. Without loss of generality, assume $p$ contains no self-intersection. As we move along $q$ let our three intersections be $v_1$, $v_2$, and $v_3$. Then $p$ and $q$ form two lenses - one with endpoints $v_1$ and $v_2$, and one with endpoints $v_2$ and $v_3$. Since these lie on opposite sides of $p$, they have disjoint interiors. Thus at least one medial lens does not contain the interior boundary vertex, a contradiction.
\end{proof}
\end{lemma}

With this series of lemmas, we can now  prove Theorem \ref{FirstMedialTheorem}.
\begin{proof}
Let $\Gamma$ be an rnpd irreducible by our moves. Part (a) holds by Lemma~\ref{NoMedialCircles}. Part (a) and Lemma~\ref{OneMedialLoop} imply that $M(\Gamma)$ contains at most one medial loop.  By Lemma~\ref{NoEmptyLenses}, if there is a medial loop, then it must contain the interior boundary vertex.  This proves part (b) of the theorem. Finally, by Lemma~\ref{NoTwoLenses}, any two distinct strands of $m(\Gamma)$ intersect at most twice. Lemma~\ref{NoEmptyLenses} tells us that any two strands that do intersect twice must contain the interior boundary vertex in the lens they create, so we have proven part (c) or the theorem.
\end{proof}

\begin{conj}
A graph is irreducible by $Y-\Delta$ moves,  pendant  removal,  self-edge  removal,  parallel reductions, series reductions, antenna jumping, and antenna absorption if and only if it has three medial strands which pairwise intersect twice, there is a medial strand that self-intersects, and it is irreducible by $Y-\Delta$ moves,  pendant  removal,  self-edge  removal,  parallel reductions, series reductions, and antenna jumping.
\end{conj}

\begin{thm}\label{thm:same-z-seq}
Two rnpds irreducible by $Y-\Delta$ moves, pendant removal, self-edge removal, parallel reductions, series reductions, and antenna jumping have the same z-sequence if and only if their graphs related by $Y-\Delta$ moves. 

\begin{proof}
Note that motions don't change the relative order of strand endpoints, or whether strands go around the interior boundary vertex clockwise or counterclockwise, so the z-sequence is preserved by motions and thus $Y-\Delta$ moves.

For the other direction, we  induct on the number of strands. If only one strand exists, the theorem holds trivially.

Suppose we have we have two rnpds $\Gamma$ and $\Gamma'$ with multiple strands and the same $z$-sequence. Each strand in $\Gamma$ which does not intersect itself divides $\Gamma$ into two regions. For strand $i$ without a self-intersection, call the region which does not contain the boundary vertex $\Gamma_{i,1}$, and the other $\Gamma_{i,2}$. Choose $s$ such that $\Gamma_{s,1}$ contains the fewest possible number of regions.  %We will call this strand $s_\Gamma$.  Let $s_\Gamma'$ be the strand in $\Gamma'$ with label $s$.
We can define $\Gamma'_{s,1}$, and $\Gamma'_{s,2}$ as we did for $\Gamma$.
%Suppose we have we have two rnpds $\Gamma$ and $\Gamma'$ with multiple strands and the same $z$-sequence. Take a strand $s$ in $\Gamma$ which does not intersect itself. Let $s'$ in $\Gamma'$ be the strand with the same label as $s$. The strand $s$ divides $\Gamma$ into two regions. Call the region which does not contain the boundary vertex $\Gamma_1$, and the other $\Gamma_2$. Define $\Gamma'_1$ and $\Gamma'_2$ similarly. Without loss of generality, we picked $s$ earlier such that $\Gamma_1$ contains the fewest possible number of regions.
Since we chose $s$ where $\Gamma_{s,1}$ is minimal, every strand in region $\Gamma_{s,1}$ intersects $s$. 

Let $G$ be the graph we obtain when we perform $Y-\Delta$ moves on $\Gamma$ to minimize the number of regions on the side of strand $s$ not containing the interior boundary vertex. Define $G_1$ and $G_2$ as the regions partitioned by strand $s$ in $G$ where $G_1$ is the side corresponding to $\Gamma_{s,1}$. Define $G'$, $G'_1$, and  $G'_2$ similarly from $\Gamma'$.

We claim  that $G_1$ contains no medial strand intersections. Suppose for contradiction that $p$ and $q$ are strand segments that lie within $G_1$ and intersect at $v_0$ in $G_1$. Let the endpoint of $p$ on $s'$ be $v_1$. Let the endpoint of $q$ on $s'$ be $v_2$.
Let the triangle formed by $v
_0$, $v_1$, and $v_2$ be $T$. By motions on triangular faces in $T$, we may make $T$ empty without increasing the number of regions in $G_1$. Then, a motion on $T$ moves $T$ to $G_2$, decreasing the number of regions in $G_1$ by one, a contradiction. Similarly, $G'_1$ contains no medial strand intersections. 

Furthermore, $G$ strongly restricted to vertices in $G_2$ is an irreducible graph whose $z$-sequence may be computed from from the $z$-sequence of $G$ by deleting strand $s$ and relabeling the remaining strands. Since we may determine from the $z$-sequence of $G'$ which strand segments of $G'_1$ intersect $s$, $G$ weakly restricted to vertices $G_1$ and $G'$ weakly restricted to vertices in $G'_1$ are isomorphic  graphs. As above, $G'$ strongly restricted to $G'_2$ is a graph irreducible by $Y-\Delta$ moves, pendant removal, self-edge removal, parallel reductions, series reductions, and antenna jumping whose $z$-sequence may be computed from from the z-sequence of $\Gamma'$ by deleting strand $s$ and relabeling the remaining strands. Then $G_2$ and $G'_2$ have the same $z$-sequence. By the inductive hypothesis, $G_2$ and $G'_2$ are $Y-\Delta$ equivalent. This means $G$ and $G'$, and therefore $\Gamma$ and $\Gamma'$, are $Y-\Delta$ equivalent as well.
\end{proof}
\end{thm}

\section{Conditional Local Moves and Spider Graphs}
%In this section, we will define a canonical family of rnpds and prove that all the graphs in this family are recoverable. We begin by defining two types of edges that will be central to this family.

\subsection{Spider Graphs}

\begin{defn}
The \textit{standard graphs}, denoted by $\Sigma_n$ for $n\geq 2$, are certain critical cprns with $z$-sequence $1,2,...,n,1,2,...,n$.  See Section 7 of~\cite{curtis_ingerman_morrow} for a detailed construction.
\end{defn}

In \cite{curtis_mooers_morrow}, it was shown that each $\Sigma_{4m + 3}$ is recoverable. An algorithm that transforms any critical cprn into some $\Sigma_{4m + 3}$ was introduced in \cite{curtis_ingerman_morrow} and was central to the proof that critical cprns are recoverable.  We define a family of cprns that are electrically equivalent to standard graphs. These graphs will be called 4-periodic graphs.  To do this, we'll first need to define some special types of edges.
\begin{defn}
An edge is a \textit{boundary edge} if it connects two boundary vertices.  An edge  is a \textit{boundary spike} if it is connected to a boundary vertex of degree 1.
\end{defn}
We will denote the 4-periodic graphs as $\Pi_n$ for $n\geq 3$. The construction is as follows: For $n\geq 3$, we start with an $n$-gon whcich we will consider to be layer 1.  Then we extend our polygon to have exactly $\floor{\frac{n+1}{4}}$ layers, where each layer is a concentric $n$-gon connected to the $n$-gon from the previous layer by $n$ edges between respective vertices (see Fig. 1). For $n\equiv 3 \mod 4$ we will have $\frac{n+1}{4}$ layers.  For $n\equiv 0 \mod 4$ we will have $\frac{n}{4}$ layers and add $\frac{n}{2}$ consecutive boundary spikes.  For $n\equiv 1 \mod 4$, we will have $\frac{n-1}{4}$ layers and add all $n$ boundary spikes. For $n\equiv 2 \mod 4$, we will have $\frac{n-2}{4}$ layers, add all $n$ boundary spikes, and then connect $\frac{n}{2}$ consecutive boundary vertices. In each case, we let all outermost vertices be boundary vertices (see Figure~\ref{fig:4-periodic}).

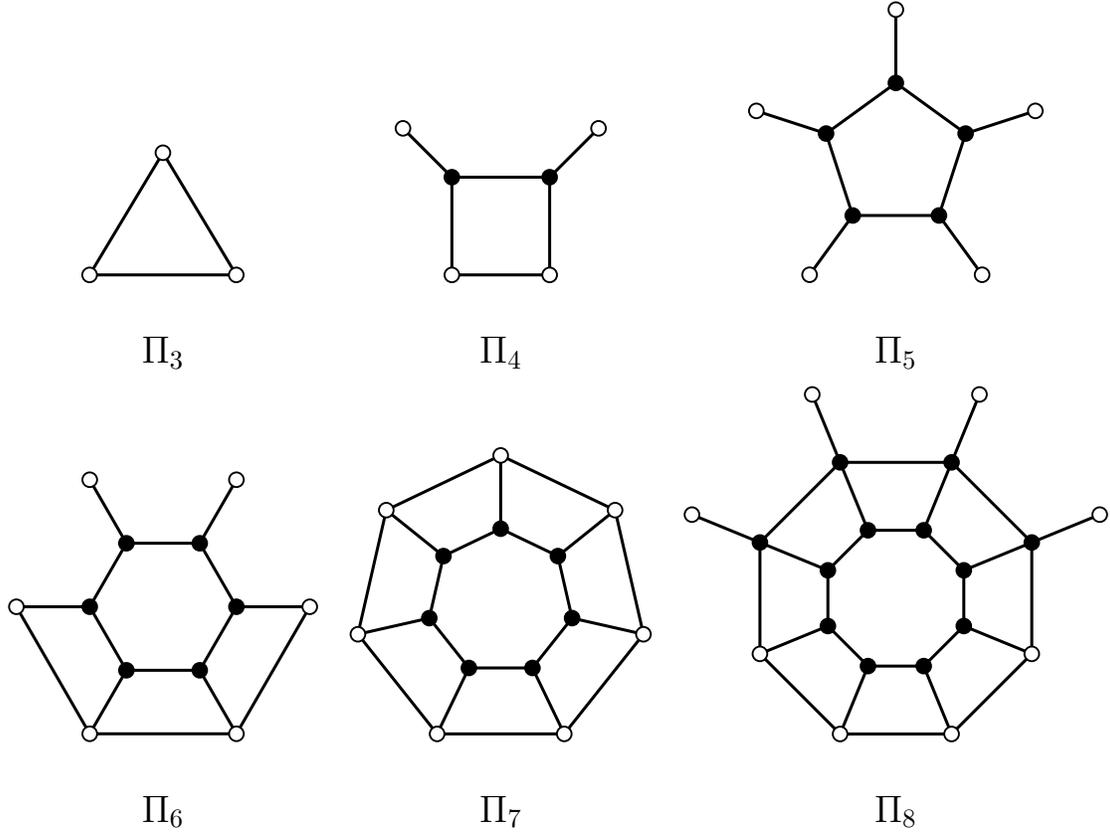
\begin{figure}[h]
\centering
\begin{tabular}{ccc}
    \begin{tikzpicture}[scale=0.65]
    \path[-, line width=0.4mm]
    (0,0) edge (1.5,2.5)
    (1.5,2.5) edge (3,0)
    (3,0) edge (0,0);
    \draw [line width=0.25mm, fill=white] (0,0) circle (1.5mm);
    \draw [line width=0.25mm, fill=white] (1.5,2.5) circle (1.5mm);
    \draw [line width=0.25mm, fill=white] (3,0) circle (1.5mm);
    \end{tikzpicture}
    &
    \begin{tikzpicture}[scale=0.65]
    \path[-, line width=0.4mm]
    (5,0) edge (7,0)
    (5,0) edge (5,2)
    (5,2) edge (7,2)
    (7,2) edge (7,0)
    (8,3) edge (7,2)
    (5,2) edge (4,3);
    \draw [line width=0.25mm, fill=white] (5,0) circle (1.5mm);
    \draw [line width=0.25mm, fill=white] (4,3) circle (1.5mm);
    \draw [line width=0.25mm, fill=black] (5,2) circle (1.5mm);
    \draw [line width=0.25mm, fill=white] (7,0) circle (1.5mm);
    \draw [line width=0.25mm, fill=black] (7,2) circle (1.5mm);
    \draw [line width=0.25mm, fill=white] (8,3) circle (1.5mm);
    \end{tikzpicture}
    &
    \begin{tikzpicture}[scale=0.65]
    \path[-, line width=0.4mm]
    (0.882,-1.213) edge (-0.882,-1.213)
    (-0.882,-1.213) edge (-1.427,0.464)
    (-1.427,0.464) edge (0,1.5)
    (0,1.5) edge (1.427,0.464)
    (1.427,0.464) edge (0.882,-1.213)
    (1.764,-2.426) edge (0.882,-1.213)
    (-1.764,-2.426) edge (-0.882,-1.213)
    (0,3) edge (0,1.5)
    (2.853,0.928) edge (1.427,0.464)
    (-2.853,0.928) edge (-1.427,0.464);
    \draw [line width=0.25mm, fill=white] (-2.853,0.928) circle (1.5mm);
    \draw [line width=0.25mm, fill=white] (2.853,0.928) circle (1.5mm);
    \draw [line width=0.25mm, fill=black] (0,1.5) circle (1.5mm);
    \draw [line width=0.25mm, fill=black] (0.882,-1.213) circle (1.5mm);
    \draw [line width=0.25mm, fill=black] (-0.882,-1.213) circle (1.5mm);
    \draw [line width=0.25mm, fill=black] (1.427,0.464) circle (1.5mm);
    \draw [line width=0.25mm, fill=black] (-1.427,0.464) circle (1.5mm);
    \draw [line width=0.25mm, fill=white] (0,3) circle (1.5mm);
    \draw [line width=0.25mm, fill=white] (1.764,-2.426) circle (1.5mm);
    \draw [line width=0.25mm, fill=white] (-1.764,-2.426) circle (1.5mm);
    \end{tikzpicture}
    \\
    {\Large$\Pi_3$} &
    {\Large$\Pi_4$} &
    {\Large$\Pi_5$} \\
    \begin{tikzpicture}[scale=0.65]
    \path[-, line width=0.4mm]
    (4,0) edge (5.5,0)
    (10,0) edge (8.5,0)
    (7.75,-1.3) edge (6.25,-1.3)
    (6.25,-1.3) edge (5.5,0)
    (5.5,0) edge (6.25,1.3)
    (6.25,1.3) edge (7.75,1.3)
    (7.75,1.3) edge (8.5,0)
    (8.5,0) edge (7.75,-1.3)
    (7.75,-1.3) edge (8.5,-2.6)
    (6.25,-1.3) edge (5.5,-2.6)
    (6.25,1.3) edge  (5.5,2.6)
    (7.75,1.3) edge  (8.5,2.6)
    (5.5,-2.6) edge  (8.5,-2.6)
    (5.5,-2.6) edge  (4,0)
    (8.5,-2.6) edge  (10,0);
    \draw [line width=0.25mm, fill=white] (10,0) circle (1.5mm);
    \draw [line width=0.25mm, fill=white] (4,0) circle (1.5mm);
    \draw [line width=0.25mm, fill=black] (7.75,1.3) circle (1.5mm);
    \draw [line width=0.25mm, fill=black] (7.75,-1.3) circle (1.5mm);
    \draw [line width=0.25mm, fill=black] (6.25,1.3) circle (1.5mm);
    \draw [line width=0.25mm, fill=black] (6.25,-1.3) circle (1.5mm);
    \draw [line width=0.25mm, fill=white] (8.5,2.6) circle (1.5mm);
    \draw [line width=0.25mm, fill=white] (8.5,-2.6) circle (1.5mm);
    \draw [line width=0.25mm, fill=white] (5.5,2.6) circle (1.5mm);
    \draw [line width=0.25mm, fill=white] (5.5,-2.6) circle (1.5mm);
    \draw [line width=0.25mm, fill=black] (5.5,0) circle (1.5mm);
    \draw [line width=0.25mm, fill=black] (8.5,0) circle (1.5mm);
    \end{tikzpicture}
    &
    \begin{tikzpicture}[scale=0.65]
    \path[-, line width=0.4mm]
    (0,1.5) edge (1.17, 0.94)
    (1.17, 0.94) edge (1.46,-0.33)
    (1.46,-0.33) edge (0.65,-1.35)
    (0.65,-1.35) edge (-0.65,-1.35)
    (-0.65,-1.35) edge (-1.46,-0.33)
    (-1.46,-0.33) edge (-1.17,0.94)
    (-1.17,0.94) edge (0,1.5)
    (0,3) edge (2.34, 1.88)
    (2.34, 1.88) edge (2.92,-0.66)
    (2.92,-0.66) edge (1.3,-2.7)
    (1.3,-2.7) edge (-1.3,-2.7)
    (-1.3,-2.7) edge (-2.92,-0.66)
    (-2.92,-0.66) edge (-2.34,1.88)
    (-2.34,1.88) edge (0,3)
    (0,1.5) edge (0,3)
    (1.17, 0.94) edge (2.34, 1.88)
    (1.46,-0.33) edge (2.92,-0.66)
    (0.65,-1.35) edge (1.3,-2.7)
    (-0.65,-1.35) edge (-1.3,-2.7)
    (-1.46,-0.33) edge (-2.92,-0.66)
    (-1.17,0.94) edge (-2.34,1.88);
    \draw [line width=0.25mm, fill=black] (0,1.5) circle (1.5mm);
    \draw [line width=0.25mm, fill=black] (1.17, 0.94) circle (1.5mm);
    \draw [line width=0.25mm, fill=black] (1.46,-0.33) circle (1.5mm);
    \draw [line width=0.25mm, fill=black] (0.65,-1.35) circle (1.5mm);
    \draw [line width=0.25mm, fill=black] (-0.65,-1.35) circle (1.5mm);
    \draw [line width=0.25mm, fill=black] (-1.46,-0.33) circle (1.5mm);
    \draw [line width=0.25mm, fill=black] (-1.17,0.94) circle (1.5mm);
    \draw [line width=0.25mm, fill=white] (0,3) circle (1.5mm);
    \draw [line width=0.25mm, fill=white] (2.34, 1.88) circle (1.5mm);
    \draw [line width=0.25mm, fill=white] (2.92,-0.66) circle (1.5mm);
    \draw [line width=0.25mm, fill=white] (1.3,-2.7) circle (1.5mm);
    \draw [line width=0.25mm, fill=white] (-1.3,-2.7) circle (1.5mm);
    \draw [line width=0.25mm, fill=white] (-2.92,-0.66) circle (1.5mm);
    \draw [line width=0.25mm, fill=white] (-2.34,1.88) circle (1.5mm);
    \end{tikzpicture}
    &
    \begin{tikzpicture}[scale=0.65]
    \path[-, line width=0.4mm]
    (0.57, 1.39) edge (1.39, 0.57)
    (1.39, 0.57) edge (1.39, -0.57)
    (1.39, -0.57) edge (0.57,-1.39)
    (0.57,-1.39) edge (-0.57,-1.39)
    (-0.57,-1.39) edge (-1.39, -0.57)
    (-1.39, -0.57) edge (-1.39, 0.57)
    (-1.39, 0.57) edge (-0.57, 1.39)
    (-0.57, 1.39) edge (0.57, 1.39)
    (1.14,2.78) edge (2.78, 1.14)
    (2.78, 1.14) edge (2.78, -1.14)
    (2.78, -1.14) edge (1.14,-2.78)
    (1.14,-2.78) edge (-1.14,-2.78)
    (-1.14,-2.78) edge (-2.78, -1.14)
    (-2.78, -1.14) edge (-2.78, 1.14)
    (-2.78, 1.14) edge (-1.14,2.78)
    (-1.14, 2.78) edge (1.14,2.78)
    (1.14,2.78) edge (0.57, 1.39)
    (2.78, 1.14) edge (1.39, 0.57)
    (2.78, -1.14) edge (1.39, -0.57)
    (1.14,-2.78) edge (0.57, -1.39)
    (-1.14,-2.78) edge (-0.57, -1.39)
    (-2.78, -1.14) edge (-1.39, -0.57)
    (-2.78, 1.14) edge (-1.39, 0.57)
    (-1.14, 2.78) edge (-0.57, 1.39)
    (1.71,4.17) edge (1.14,2.78)
    (4.17, 1.71) edge (2.78,1.14)
    (-4.17, 1.71) edge (-2.78,1.14)
    (-1.71,4.17) edge (-1.14,2.78);
    \draw [line width=0.25mm, fill=black] (0.57,1.39) circle (1.5mm);
    \draw [line width=0.25mm, fill=black] (1.39, 0.57) circle (1.5mm);
    \draw [line width=0.25mm, fill=black] (1.39, -0.57) circle (1.5mm);
    \draw [line width=0.25mm, fill=black] (0.57,-1.39) circle (1.5mm);
    \draw [line width=0.25mm, fill=black] (-0.57,-1.39) circle (1.5mm);
    \draw [line width=0.25mm, fill=black] (-1.39, -0.57) circle (1.5mm);
    \draw [line width=0.25mm, fill=black] (-1.39, 0.57) circle (1.5mm);
    \draw [line width=0.25mm, fill=black] (-0.57,1.39) circle (1.5mm);
    \draw [line width=0.25mm, fill=black] (1.14,2.78) circle (1.5mm);
    \draw [line width=0.25mm, fill=black] (2.78, 1.14) circle (1.5mm);
    \draw [line width=0.25mm, fill=white] (2.78, -1.14) circle (1.5mm);
    \draw [line width=0.25mm, fill=white] (1.14,-2.78) circle (1.5mm);
    \draw [line width=0.25mm, fill=white] (-1.14,-2.78) circle (1.5mm);
    \draw [line width=0.25mm, fill=white] (-2.78, -1.14) circle (1.5mm);
    \draw [line width=0.25mm, fill=black] (-2.78, 1.14) circle (1.5mm);
    \draw [line width=0.25mm, fill=black] (-1.14,2.78) circle (1.5mm);
    \draw [line width=0.25mm, fill=white] (1.71,4.17) circle (1.5mm);
    \draw [line width=0.25mm, fill=white] (4.17, 1.71) circle (1.5mm);
    \draw [line width=0.25mm, fill=white] (-4.17, 1.71) circle (1.5mm);
    \draw [line width=0.25mm, fill=white] (-1.71,4.17) circle (1.5mm);
    \end{tikzpicture}
    \\
    {\Large$\Pi_6$} &
    {\Large$\Pi_7$} &
    {\Large$\Pi_8$} \\
\end{tabular}
\caption{The first six 4-periodic cprns, $\Pi_n$ for $3\leq n\leq 8$.}
\label{fig:4-periodic}
\end{figure}

\begin{defn}
In $\Pi_n$ for $n$ even, we can see that as we go around the boundary of the disk, we have a set of boundary vertices with boundary spikes and then a set of boundary vertices connected by boundary edges.  Let the first and last of the boundary vertices connected by boundary edges be $v$ and $w$.  For $n\equiv 0\mod 4$, each of these vertices is of degree 2.  We will refer to the edges connected to $v$ and $w$ that are not boundary edges as \emph{boundary pseudo-spikes}.  For $n\equiv 2\mod 4$, each of these vertices is of degree 3.  The vertices have a boundary edge, an edge going inward, and an edge connected to an internal vertex that is connected to a boundary spike.  We will refer to the last of these edges as a \emph{pseudo-boundary edge}.
\end{defn}

\begin{figure}[h]
\centering
\begin{tabular}{ccc}
    \begin{tikzpicture}[scale=0.65]
    \path[-, line width=0.5mm]
    (5,2) edge (7,2);
    \path[-, line width=0.5mm, violet]
    (5,0) edge (7,0);
    \path[-, line width=0.5mm, magenta]
    (8,3) edge (7,2)
    (5,2) edge (4,3);
    \path[-, line width=0.5mm, cyan]
    (5,0) edge (5,2)
    (7,2) edge (7,0);
    \draw [line width=0.25mm, fill=white] (5,0) circle (1.5mm);
    \draw [line width=0.25mm, fill=white] (4,3) circle (1.5mm);
    \draw [line width=0.25mm, fill=black] (5,2) circle (1.5mm);
    \draw [line width=0.25mm, fill=white] (7,0) circle (1.5mm);
    \draw [line width=0.25mm, fill=black] (7,2) circle (1.5mm);
    \draw [line width=0.25mm, fill=white] (8,3) circle (1.5mm);
    \end{tikzpicture}
    &
    \begin{tikzpicture}[scale=0.65]
    \path[-, line width=0.5mm]
    (7.75,-1.3) edge (6.25,-1.3)
    (6.25,-1.3) edge (5.5,0)
    (5.5,0) edge (6.25,1.3)
    (6.25,1.3) edge (7.75,1.3)
    (7.75,1.3) edge (8.5,0)
    (8.5,0) edge (7.75,-1.3)
    (7.75,-1.3) edge (8.5,-2.6)
    (6.25,-1.3) edge (5.5,-2.6);
    \path[-, line width=0.5mm, violet]
    (5.5,-2.6) edge  (8.5,-2.6)
    (5.5,-2.6) edge  (4,0)
    (8.5,-2.6) edge  (10,0);
    \path[-, line width=0.5mm, magenta]
    (6.25,1.3) edge  (5.5,2.6)
    (7.75,1.3) edge  (8.5,2.6);
    \path[-, line width=0.5mm, orange]
    (4,0) edge (5.5,0)
    (10,0) edge (8.5,0);
    \draw [line width=0.25mm, fill=white] (10,0) circle (1.5mm);
    \draw [line width=0.25mm, fill=white] (4,0) circle (1.5mm);
    \draw [line width=0.25mm, fill=black] (7.75,1.3) circle (1.5mm);
    \draw [line width=0.25mm, fill=black] (7.75,-1.3) circle (1.5mm);
    \draw [line width=0.25mm, fill=black] (6.25,1.3) circle (1.5mm);
    \draw [line width=0.25mm, fill=black] (6.25,-1.3) circle (1.5mm);
    \draw [line width=0.25mm, fill=white] (8.5,2.6) circle (1.5mm);
    \draw [line width=0.25mm, fill=white] (8.5,-2.6) circle (1.5mm);
    \draw [line width=0.25mm, fill=white] (5.5,2.6) circle (1.5mm);
    \draw [line width=0.25mm, fill=white] (5.5,-2.6) circle (1.5mm);
    \draw [line width=0.25mm, fill=black] (5.5,0) circle (1.5mm);
    \draw [line width=0.25mm, fill=black] (8.5,0) circle (1.5mm);
    \end{tikzpicture}
    &
    \begin{tikzpicture}[scale=0.65]
    \path[-, line width=0.5mm]
    (0.57, 1.39) edge (1.39, 0.57)
    (1.39, 0.57) edge (1.39, -0.57)
    (1.39, -0.57) edge (0.57,-1.39)
    (0.57,-1.39) edge (-0.57,-1.39)
    (-0.57,-1.39) edge (-1.39, -0.57)
    (-1.39, -0.57) edge (-1.39, 0.57)
    (-1.39, 0.57) edge (-0.57, 1.39)
    (-0.57, 1.39) edge (0.57, 1.39)
    (1.14,2.78) edge (2.78, 1.14)
    (-2.78, 1.14) edge (-1.14,2.78)
    (-1.14, 2.78) edge (1.14,2.78)
    (1.14,2.78) edge (0.57, 1.39)
    (2.78, 1.14) edge (1.39, 0.57)
    (2.78, -1.14) edge (1.39, -0.57)
    (1.14,-2.78) edge (0.57, -1.39)
    (-1.14,-2.78) edge (-0.57, -1.39)
    (-2.78, -1.14) edge (-1.39, -0.57)
    (-2.78, 1.14) edge (-1.39, 0.57)
    (-1.14, 2.78) edge (-0.57, 1.39);
    \path[-, line width=0.5mm, cyan]
    (2.78, 1.14) edge (2.78, -1.14)
    (-2.78, -1.14) edge (-2.78, 1.14);
    \path[-, line width=0.5mm, magenta]
    (1.71,4.17) edge (1.14,2.78)
    (4.17, 1.71) edge (2.78,1.14)
    (-4.17, 1.71) edge (-2.78,1.14)
    (-1.71,4.17) edge (-1.14,2.78);
    \path[-, line width=0.5mm, violet]
    (2.78, -1.14) edge (1.14,-2.78)
    (1.14,-2.78) edge (-1.14,-2.78)
    (-1.14,-2.78) edge (-2.78, -1.14);
    \draw [line width=0.25mm, fill=black] (0.57,1.39) circle (1.5mm);
    \draw [line width=0.25mm, fill=black] (1.39, 0.57) circle (1.5mm);
    \draw [line width=0.25mm, fill=black] (1.39, -0.57) circle (1.5mm);
    \draw [line width=0.25mm, fill=black] (0.57,-1.39) circle (1.5mm);
    \draw [line width=0.25mm, fill=black] (-0.57,-1.39) circle (1.5mm);
    \draw [line width=0.25mm, fill=black] (-1.39, -0.57) circle (1.5mm);
    \draw [line width=0.25mm, fill=black] (-1.39, 0.57) circle (1.5mm);
    \draw [line width=0.25mm, fill=black] (-0.57,1.39) circle (1.5mm);
    \draw [line width=0.25mm, fill=black] (1.14,2.78) circle (1.5mm);
    \draw [line width=0.25mm, fill=black] (2.78, 1.14) circle (1.5mm);
    \draw [line width=0.25mm, fill=white] (2.78, -1.14) circle (1.5mm);
    \draw [line width=0.25mm, fill=white] (1.14,-2.78) circle (1.5mm);
    \draw [line width=0.25mm, fill=white] (-1.14,-2.78) circle (1.5mm);
    \draw [line width=0.25mm, fill=white] (-2.78, -1.14) circle (1.5mm);
    \draw [line width=0.25mm, fill=black] (-2.78, 1.14) circle (1.5mm);
    \draw [line width=0.25mm, fill=black] (-1.14,2.78) circle (1.5mm);
    \draw [line width=0.25mm, fill=white] (1.71,4.17) circle (1.5mm);
    \draw [line width=0.25mm, fill=white] (4.17, 1.71) circle (1.5mm);
    \draw [line width=0.25mm, fill=white] (-4.17, 1.71) circle (1.5mm);
    \draw [line width=0.25mm, fill=white] (-1.71,4.17) circle (1.5mm);
    \end{tikzpicture}
    \\
\end{tabular}
\caption{Types of edges in $\Pi_4, \Pi_6$, and $\Pi_8$.  Boundary edges appear in purple, boundary spikes in pink, pseudo-boundary edges in light blue, and boundary pseudo-spikes in orange.}
\label{fig:boundary-pseudo-spike-pseudo-boundary-spike}
\end{figure}
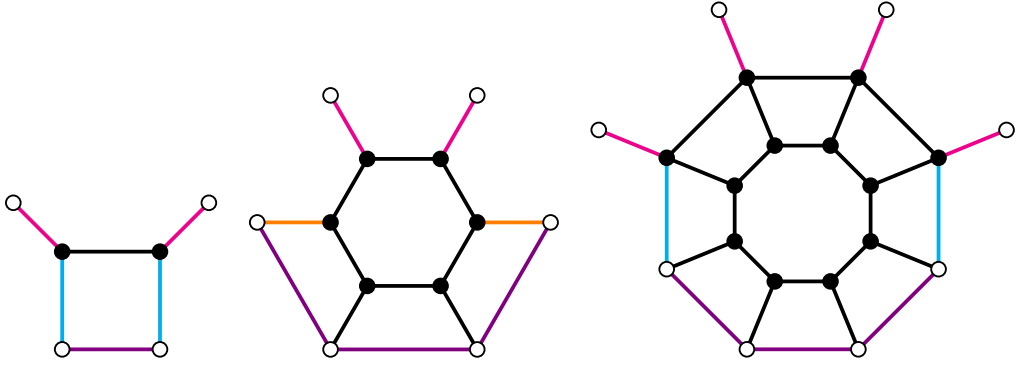

\begin{lemma}\label{4-periodic}
Let $\Pi_n$ be a 4-periodic graph and let $\ell$ be the number of layers in this graph. Let $s$ be a medial strand in $\Pi_n$ with endpoints $s_1$ and $s_2$ such that the center face of the graph is to the right of the strand when going from $s_1$ to $s_2$.  Then the number of strand endpoints between $s_1$ and $s_2$ to the left of the strand is:
\begin{enumerate}
    \item $4\ell-2$ if $n\equiv 3\mod 4$.
    \item $4\ell-1$ if $n\equiv 0\mod 4$.
    \item $4\ell$ if $n\equiv 1\mod 4$.
    \item $4\ell+1$ if $n\equiv 2\mod 4$.
\end{enumerate}
\end{lemma}
\begin{proof}
Suppose the lemma is true for case (3) for a graph with $\ell-1$ layers.  This means that given any strand $s$, the number of strand endpoints between $s_1$ and $s_2$ is $4\ell-4$ in such a graph.  Choose any strand $s$.  This strand partitions the graph into two pieces.  One of these pieces contains most of the center face.  Add extra vertices and edges to the graph so that the graph still has $\ell-1$ layers, but the underlying polygon now has 4 more sides.  We will assume without loss of generality that these vertices and edges are added in the piece of the graph that contains most of the center face.  This means it won't affect $s$ at all.  Now we add a layer to get to the next graph that falls under case (3).  We can observe that adding the layer forces us to have two extra sides of the polygon between $s_1$ and $s_2$, which corresponds to exactly 4 extra strand endpoints (see Figure~\ref{fig:add-layer}). Thus, the number of strand endpoints between $s_1$ and $s_2$ is $4\ell-4+4=4\ell$, as desired, and we have proven case (3) of the lemma.  The proof for case (1) is similar.
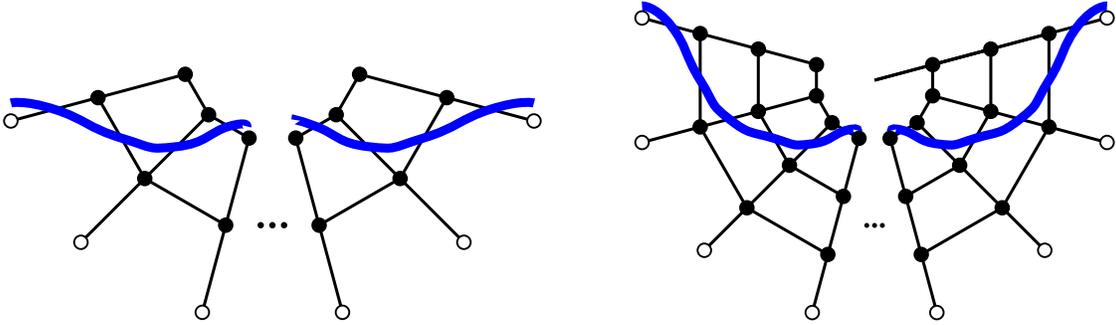
\begin{figure}[htp]
    \centering
    \begin{tikzpicture}[scale=0.6]
        %\draw [line width=0.5mm,blue] plot [smooth, tension=1] coordinates { (4,4) (2.5,3.4) (2.5,1.6) (4,1) };
        %Define coordinates here - using 6 vertices in n-gon
        \def\a{1.93185}
        \def\b{1.41421}
        \def\c{.517638}
        %Draw 2 layers of n-gon
        \draw[line width=0.4mm] (-2*\a, -2*\c) -- (-2*\b, -2* \b) -- (-2 * \c, -2 *\a);
        \draw[line width=0.4mm] (2*\a, -2*\c) -- (2*\b, -2* \b) -- (2 * \c, -2 *\a);
        \draw[line width=0.4mm] (-1*\a, -1*\c) -- (-1*\b, -1* \b) -- (-1 * \c, -1 *\a);
        \draw[line width=0.4mm] (1*\a, -1*\c) -- (1*\b, -1* \b) -- (1 * \c, -1 *\a);
        %Connect the layers
        \draw[line width=0.4mm] (-3*\a, -3*\c) -- (-2*\a, -2*\c) -- (-1*\a, -1*\c);
        \draw[line width=0.4mm] (-3*\b, -3*\b) -- (-2*\b, -2* \b) -- (-1*\b, -1* \b);
        \draw[line width=0.4mm] (-3 *\c, -3 * \a) -- (-2 * \c, -2 *\a) -- (-1 *\c, -1 * \a);
        \draw[line width=0.4mm] (3*\a, -3*\c) -- (2*\a, -2*\c) -- (1*\a, -1*\c);
        \draw[line width=0.4mm] (3 * \b, -3*\b) -- (2*\b, -2* \b) -- (1*\b, -1* \b);
        \draw[line width=0.4mm] (3*\c, -3*\a) -- (2 * \c, -2 *\a) -- (1 *\c, -1 * \a);
        
        %Draw medial strands
        \draw [line width=1.2mm, blue] plot [smooth, tension=1] coordinates {(-3*\a, -3*\c + .4) (-2.5*\a, -2.5*\c) (-\a - \b, -\c - \b ) (-1.5 * \b, -1.5 * \b) (-.5 * \b - .5* \c , -0.5*\a - .5*\b ) (-.5 * \b - .5* \c + .4, -0.5*\a - .5*\b)};
        \draw [line width=1.2mm,blue] plot [smooth, tension=1] coordinates { (3*\a, -3*\c + .4) (2.5*\a, -2.5*\c) (\a + \b, -\c - \b ) (1.5 * \b, -1.5 * \b) (.5 * \b + .5* \c , -0.5*\a - .5*\b) (\c, -1*\a + .4)};
        %Draw ...'s
        \draw[line width = .25mm, fill=black](0, -2*\a) circle (.5mm);
        \draw[line width = .25mm, fill=black](.5*\c, -2*\a) circle (.5mm);
        \draw[line width = .25mm, fill=black](-.5*\c, -2*\a) circle (.5mm);
        %Add in vertices
        \draw [line width=0.25mm, fill=black] (2*\c, -2*\a) circle (1.5mm);
        \draw [line width=0.25mm, fill=white] (3*\c, -3*\a) circle (1.5mm);
        \draw [line width=0.25mm, fill=black] (\c, -1*\a) circle (1.5mm);
        \draw [line width=0.25mm, fill=black] (2*\b, -2*\b) circle (1.5mm);
        \draw [line width=0.25mm, fill=white] (3*\b, -3*\b) circle (1.5mm);
        \draw [line width=0.25mm, fill=black] (\b, -1 *\b) circle (1.5mm);
        \draw [line width=0.25mm, fill=black] (2*\a, -2*\c) circle (1.5mm);
        \draw [line width=0.25mm, fill=white] (3*\a, -3*\c) circle (1.5mm);   
        \draw [line width=0.25mm, fill=black] (\a, -1*\c) circle (1.5mm);
        \draw [line width=0.25mm, fill=black] (-2*\a, -2*\c) circle (1.5mm);
        \draw [line width=0.25mm, fill=white] (-3*\a, -3*\c) circle (1.5mm);
        \draw [line width=0.25mm, fill=black] (-1*\a, -1*\c) circle (1.5mm);
        \draw [line width=0.25mm, fill=black] (-2*\b, -2*\b) circle (1.5mm);
        \draw [line width=0.25mm, fill=white] (-3*\b, -3*\b) circle (1.5mm);
        \draw [line width=0.25mm, fill=black] (-1*\b, -1*\b) circle (1.5mm);
        \draw [line width=0.25mm, fill=black] (-2*\c, -2*\a) circle (1.5mm);
        \draw [line width=0.25mm, fill=white] (-3*\c, -3*\a) circle (1.5mm);
        \draw [line width=0.25mm, fill=black] (-1 *\c, -1* \a) circle (1.5mm);
    \end{tikzpicture}
    \hspace{1cm}
    %SECOND PICTURE - ADDING IN A LAYER
     \begin{tikzpicture}[scale=0.4]
        %\draw [line width=0.5mm,blue] plot [smooth, tension=1] coordinates { (4,4) (2.5,3.4) (2.5,1.6) (4,1) };
        %Define coordinates here - using 6 vertices in n-gon
        \def\a{1.93185}
        \def\b{1.41421}
        \def\c{.517638}
        %Draw 2 layers of n-gon - Fix first cordinate
        \draw[line width=0.4mm] (-3 * \a, 3*\c) -- (-3*\a, -3*\c) -- (-3*\b, -3* \b) -- (-3 * \c, -3 *\a);
        \draw[line width=0.4mm] (-2 * \a, 2*\c) -- (-2*\a, -2*\c) -- (-2*\b, -2* \b) -- (-2 * \c, -2 *\a);
        \draw[line width=0.4mm]  -- (2 * \a, 2*\c)-- (2*\a, -2*\c) -- (2*\b, -2* \b) -- (2 * \c, -2 *\a);
         \draw[line width=0.4mm]  -- (3 * \a, 3*\c)-- (3*\a, -3*\c) -- (3*\b, -3* \b) -- (3 * \c, -3 *\a);
        \draw [line width=0.4mm](-1 * \a, \c) -- (-1*\a, -1*\c) -- (-1*\b, -1* \b) -- (-1 * \c, -1 *\a);
        \draw [line width=0.4mm](\a, \c) -- (1*\a, -1*\c) -- (1*\b, -1* \b) -- (1 * \c, -1 *\a);
        %Connect the layers
        \draw [line width=0.4mm](-4 * \a, 4 *\c) -- (-3*\a, 3*\c) -- (-1*\a, 1*\c);
        \draw [line width=0.4mm](4 *\a, 4 *\c) -- (3*\a, 3*\c) -- (1*\a, \c);
        \draw [line width=0.4mm](-4*\a, -4*\c) -- (-3*\a, -3*\c) -- (-1*\a, -1*\c);
        \draw [line width=0.4mm](-4*\b, -4 * \b) --  (-3*\b, -3* \b) -- (-1*\b, -1* \b);
        \draw[line width=0.4mm] (-4 * \c, -4 * \a) -- (-3 * \c, -3 *\a) -- (-1 *\c, -1 * \a);
        \draw [line width=0.4mm](4 * \a, -4 * \c) --   (3*\a, -3*\c) -- (1*\a, -1*\c);
        \draw[line width=0.4mm] (4 * \b, -4 * \b) -- (3*\b, -3* \b) -- (1*\b, -1* \b);
        \draw [line width=0.4mm](4 * \c, -4 * \a) --  (3 * \c, -3 *\a) -- (1 *\c, -1 * \a);
        %Draw medial strands
        \draw [line width=1.2mm, blue] plot [smooth, tension=1] coordinates {(-4*\a, 4*\c + .4)(-3.5*\a, 3.5*\c) (-3 * \a, 0) (-2.5 *\a, -2.5 * \c) (-\a - \b, -\c - \b ) (-1.5 * \b, -1.5 * \b) (-.5 * \b - .5* \c, -0.5*\a - .5*\b ) (-.5 * \b - .5* \c + .4, -0.5*\a - .5*\b) };
        \draw [line width=1.2mm,blue] plot [smooth, tension=1] coordinates {(4 *\a, 4 * \c + .4) (3.5 *\a, 3.5 * \c) (3 * \a, 0) (2.5 *\a, -2.5 * \c) (\a + \b , -\c - \b ) (1.5 * \b, -1.5 * \b) (.5 * \b + .5* \c , -0.5*\a - .5*\b) (.5 * \b + .5* \c - .4, -0.5*\a - .5*\b) };
        %Draw ...'s
        \draw[line width = .25mm, fill=black](0, -2.5*\a) circle (.5mm);
        \draw[line width = .25mm, fill=black](.5*\c, -2.5*\a) circle (.5mm);
        \draw[line width = .25mm, fill=black](-.5*\c, -2.5*\a) circle (.5mm);
        %Add in vertices
        \draw[line width = .25mm, fill=black] (-2*\a, 2*\c) circle (2.25mm);
        \draw[line width = .25mm, fill=black] (-3*\a, 3*\c) circle (2.25mm);
        \draw[line width = .25mm, fill=white] (-4*\a, 4*\c) circle (2.25mm);
        \draw[line width = .25mm, fill=black] (-1*\a, 1*\c) circle (2.25mm);
        \draw[line width = .25mm, fill=black] (2*\a, 2*\c) circle (2.25mm);
        \draw[line width = .25mm, fill=black] (3*\a, 3*\c) circle (2.25mm);
        \draw[line width = .25mm, fill=white] (4*\a, 4*\c) circle (2.25mm);  
        \draw[line width = .25mm, fill=black] (1*\a, 1*\c) circle (2.25mm);
        \draw [line width=0.25mm, fill=black] (3*\c, -3*\a) circle (2.25mm);
        \draw [line width=0.25mm, fill=white] (4*\c, -4*\a) circle (2.25mm);
        \draw [line width=0.25mm, fill=black] (2*\c, -2*\a) circle (2.25mm);
        \draw [line width=0.25mm, fill=black] (3*\b, -3*\b) circle (2.25mm);
        \draw [line width=0.25mm, fill=white] (4*\b, -4*\b) circle (2.25mm);
        \draw [line width=0.25mm, fill=black] (2*\b, -2 *\b) circle (2.25mm);
        \draw [line width=0.25mm, fill=black] (1*\c, -1*\a) circle (2.25mm);
        \draw [line width=0.25mm, fill=black] (3*\a, -3*\c) circle (2.25mm);
        \draw [line width=0.25mm, fill=white] (4*\a, -4*\c) circle (2.25mm);
        \draw [line width=0.25mm, fill=black] (2*\a, -2*\c) circle (2.25mm);
        \draw [line width=0.25mm, fill=black] (1*\a, -1*\c) circle (2.25mm);
        \draw [line width=0.25mm, fill=black] (-1*\c, -1*\a) circle (2.25mm);
        \draw [line width=0.25mm, fill=black] (-3*\a, -3*\c) circle (2.25mm);
        \draw [line width=0.25mm, fill=white] (-4*\a, -4*\c) circle (2.25mm);
        \draw [line width=0.25mm, fill=black] (-2*\a, -2*\c) circle (2.25mm);
        \draw [line width=0.25mm, fill=black] (-1*\a, -1*\c) circle (2.25mm);
        \draw [line width=0.25mm, fill=black] (-3*\b, -3*\b) circle (2.25mm);
        \draw [line width=0.25mm, fill=white] (-4*\b, -4*\b) circle (2.25mm);
        \draw [line width=0.25mm, fill=black] (-2*\b, -2*\b) circle (2.25mm);
        \draw [line width=0.25mm, fill=black] (-1*\b, -1*\b) circle (2.25mm);
        \draw [line width=0.25mm, fill=black] (-3*\c, -3*\a) circle (2.25mm);
        \draw [line width=0.25mm, fill=white] (-4*\c, -4*\a) circle (2.25mm);
        \draw [line width=0.25mm, fill=black] (-2 *\c, -2* \a) circle (2.25mm);
        \draw [line width=0.25mm, fill=black] (1 *\b, -1* \b) circle (2.25mm);
    \end{tikzpicture}
    \caption{Adding a layer adds two extra boundary spikes between strand endpoints.}
    \label{fig:add-layer}
\end{figure}

Now suppose that $n\equiv 0\mod 4$. We claim that each strand must have one end crossing a boundary spike and the other end crossing an boundary edge or pseudo-boundary edge. To prove this, suppose for contradiction that the strand $s$ starts by crosses a boundary spike on both ends (the other case is similar). Then the strand looks exactly like a strand from case (3) and there must be $4\ell$ strand endpoints between $s_1$ and $s_2$ to the left. This means that there are at least $\frac{4\ell-2}{2}+2$ boundary spikes in the graph ($4\ell-2$ internal strands and 2 spikes that $S$ crosses). Since $n\equiv 0\mod 4$, we get the following relation:
$$\frac{4\ell-2}{2}+2=2\ell+1=2\floor*{\frac{n+1}{4}}+1=2\cdot\frac{n}{4}+1=\frac{n+2}{2}>\frac{n}{2}$$ \\
This is a contradiction, since the number of boundary spikes in such graphs is exactly $\frac{n}{2}$. Now we will prove that in any such graph, any strand that has one end crossing a boundary spike and one end crossing a boundary edge or pseudo-boundary edge has exactly $2\ell-1$ strand endpoints between its endpoints. Assume that the lemma holds for a graph with $\ell-1$ layers. Let $s$ be a strand in the graph.  We can increase the size of the polygons in each layer as we did for the proof of case (3) without affecting $s$.  Again, adding a layer adds exactly two sides of the polygon between $s_1$ and $s_2$, and consequently 4 more strand endpoints. The proof for case (4) is similar, where we use boundary edges instead of boundary spikes and boundary spikes and boundary pseudo-spikes instead of boundary edges and pseudo-boundary edges.\\
\end{proof}

\begin{cor}
For $n\geq 3$, $\Pi_n$ is electrically equivalent to $\Sigma_n$.
\end{cor}
\begin{proof}
Let $n\equiv 1\mod 4$. We first note that the number of medial strands in $\Pi_n$ is $n$. Also, the number of endpoints is $2n$, as every boundary vertex corresponds to exactly two endpoints. Take any strand $s$ with endpoints $s_1,s_2$. By Lemma \ref{4-periodic}, we know that the number of strands between $s_1$ and $s_2$ is $4\ell$ on one side. Since $n\equiv 1\mod 4$, we know that $4\ell=4\floor{\frac{n+1}{4}}=4\cdot\frac{n-1}{4}=n-1$, so there are $n-1$ strand endpoints between $s_1$ and $s_2$. Since there are $2n-2$ endpoints other than $s_1$ and $s_2$, it follows that the number of endpoints between $s_1$ and $s_2$ on the other side is also $n-1$. Since $s$ was an arbitrary strand, the graph has the same $z$-sequence as the standard graph. Also note that the medial graph is lenseless.  This means $\Pi_n$ is electrically equivalent to $\Sigma_n$. The cases when $n\equiv 2\mod 4$, $n\equiv 3\mod 4$, and $n\equiv 0\mod 4$ are similar.
\end{proof}

\begin{defn}
A \textit{spider graph}, $\Xi_n$, is defined for each $n \geq 3$. $\Xi_n$ is constructed by placing a boundary vertex inside the center face of $\Pi_n$, and connecting it to each vertex on the boundary of that face.
\end{defn}

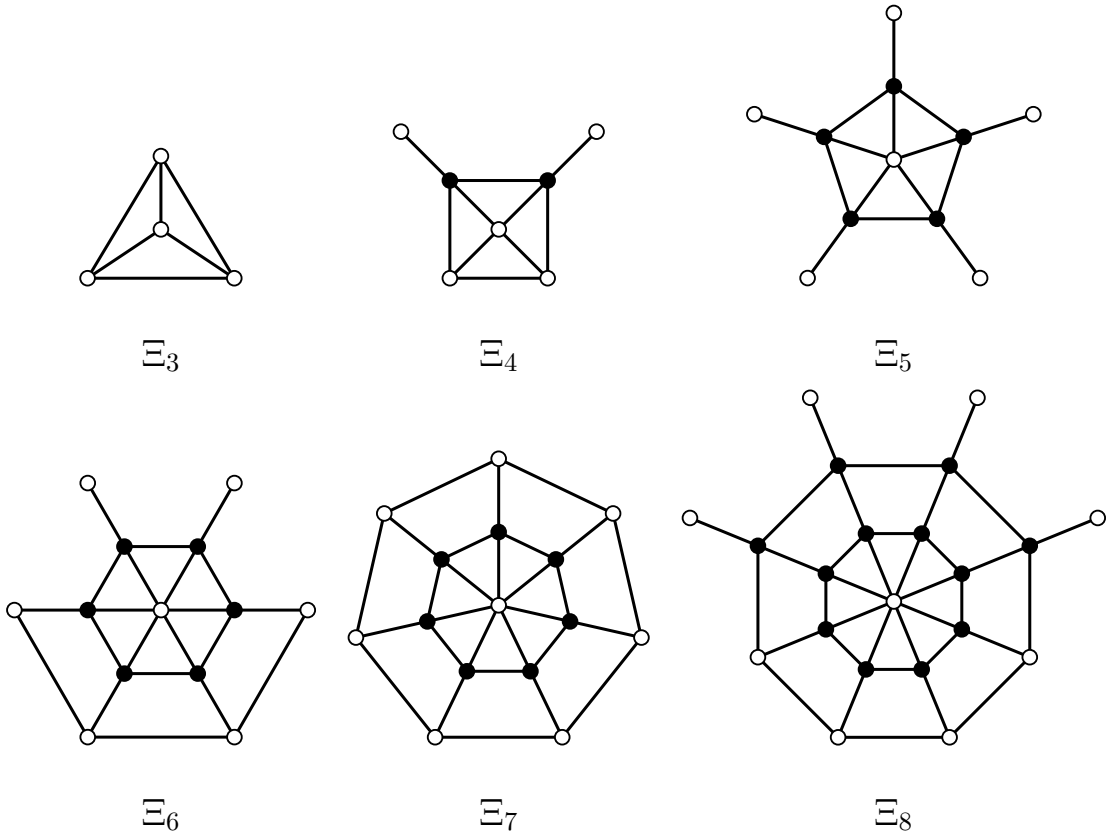
\begin{figure}[h]
    \centering
    \begin{tabular}{ccc}
    \begin{tikzpicture}[scale=0.65]
    \path[-, line width=0.4mm]
    (0,0) edge (1.5,1)
    (1.5,2.35) edge (1.5,1)
    (3,0) edge (1.5,1)
    (0,0) edge (1.5,2.5)
    (1.5,2.5) edge (3,0)
    (3,0) edge (0,0);
    \draw [line width=0.25mm, fill=white] (0,0) circle (1.5mm);
    \draw [line width=0.25mm, fill=white] (1.5,2.5) circle (1.5mm);
    \draw [line width=0.25mm, fill=white] (3,0) circle (1.5mm);
    \draw [line width=0.25mm, fill=white] (1.5,1) circle (1.5mm);
    \end{tikzpicture}
    &
    \begin{tikzpicture}[scale=0.65]
    \path[-, line width=0.4mm]
    (5,0) edge (6,1)
    (5,2) edge (6,1)
    (7,0) edge (6,1)
    (5,0) edge (7,0)
    (5,0) edge (5,2)
    (7,2) edge (6,1)
    (5,2) edge (7,2)
    (7,2) edge (7,0)
    (8,3) edge (7,2)
    (5,2) edge (4,3);
    \draw [line width=0.25mm, fill=white] (5,0) circle (1.5mm);
    \draw [line width=0.25mm, fill=white] (4,3) circle (1.5mm);
    \draw [line width=0.25mm, fill=black] (5,2) circle (1.5mm);
    \draw [line width=0.25mm, fill=white] (7,0) circle (1.5mm);
    \draw [line width=0.25mm, fill=black] (7,2) circle (1.5mm);
    \draw [line width=0.25mm, fill=white] (6,1) circle (1.5mm);
    \draw [line width=0.25mm, fill=white] (8,3) circle (1.5mm);
    \end{tikzpicture}
    &
    \begin{tikzpicture}[scale=0.65]
    \path[-, line width=0.4mm]
    (0,0) edge (0.882,-1.213)
    (0,0) edge (-0.882,-1.213)
    (0,0) edge (-1.427,0.464)
    (0,0) edge (0,1.5)
    (0,0) edge (1.427,0.464)
    (0.882,-1.213) edge (-0.882,-1.213)
    (-0.882,-1.213) edge (-1.427,0.464)
    (-1.427,0.464) edge (0,1.5)
    (0,1.5) edge (1.427,0.464)
    (1.427,0.464) edge (0.882,-1.213)
    (1.764,-2.426) edge (0.882,-1.213)
    (-1.764,-2.426) edge (-0.882,-1.213)
    (0,3) edge (0,1.5)
    (2.853,0.928) edge (1.427,0.464)
    (-2.853,0.928) edge (-1.427,0.464);
    \draw [line width=0.25mm, fill=white] (-2.853,0.928) circle (1.5mm);
    \draw [line width=0.25mm, fill=white] (2.853,0.928) circle (1.5mm);
    \draw [line width=0.25mm, fill=black] (0,1.5) circle (1.5mm);
    \draw [line width=0.25mm, fill=black] (0.882,-1.213) circle (1.5mm);
    \draw [line width=0.25mm, fill=black] (-0.882,-1.213) circle (1.5mm);
    \draw [line width=0.25mm, fill=black] (1.427,0.464) circle (1.5mm);
    \draw [line width=0.25mm, fill=black] (-1.427,0.464) circle (1.5mm);
    \draw [line width=0.25mm, fill=white] (0,0) circle (1.5mm);
    \draw [line width=0.25mm, fill=white] (0,3) circle (1.5mm);
    \draw [line width=0.25mm, fill=white] (1.764,-2.426) circle (1.5mm);
    \draw [line width=0.25mm, fill=white] (-1.764,-2.426) circle (1.5mm);
    \end{tikzpicture}
    \\
    {\Large$\Xi_3$} &
    {\Large$\Xi_4$} &
    {\Large$\Xi_5$} \\
    \begin{tikzpicture}[scale=0.65]
    \path[-, line width=0.4mm]
    (7,0) edge (7.75,-1.3)
    (7,0) edge (6.25,-1.3)
    (7,0) edge (5.5,0)
    (4,0) edge (5.5,0)
    (7,0) edge (6.25,1.3)
    (7,0) edge (7.75,1.3)
    (7,0) edge (8.5,0)
    (10,0) edge (8.5,0)
    (7.75,-1.3) edge (6.25,-1.3)
    (6.25,-1.3) edge (5.5,0)
    (5.5,0) edge (6.25,1.3)
    (6.25,1.3) edge (7.75,1.3)
    (7.75,1.3) edge (8.5,0)
    (8.5,0) edge (7.75,-1.3)
    (7.75,-1.3) edge (8.5,-2.6)
    (6.25,-1.3) edge (5.5,-2.6)
    (6.25,1.3) edge  (5.5,2.6)
    (7.75,1.3) edge  (8.5,2.6)
    (5.5,-2.6) edge  (8.5,-2.6)
    (5.5,-2.6) edge  (4,0)
    (8.5,-2.6) edge  (10,0);
    \draw [line width=0.25mm, fill=white] (10,0) circle (1.5mm);
    \draw [line width=0.25mm, fill=white] (4,0) circle (1.5mm);
    \draw [line width=0.25mm, fill=black] (7.75,1.3) circle (1.5mm);
    \draw [line width=0.25mm, fill=black] (7.75,-1.3) circle (1.5mm);
    \draw [line width=0.25mm, fill=black] (6.25,1.3) circle (1.5mm);
    \draw [line width=0.25mm, fill=black] (6.25,-1.3) circle (1.5mm);
    \draw [line width=0.25mm, fill=white] (8.5,2.6) circle (1.5mm);
    \draw [line width=0.25mm, fill=white] (8.5,-2.6) circle (1.5mm);
    \draw [line width=0.25mm, fill=white] (5.5,2.6) circle (1.5mm);
    \draw [line width=0.25mm, fill=white] (5.5,-2.6) circle (1.5mm);
    \draw [line width=0.25mm, fill=black] (5.5,0) circle (1.5mm);
    \draw [line width=0.25mm, fill=white] (7,0) circle (1.5mm);
    \draw [line width=0.25mm, fill=black] (8.5,0) circle (1.5mm);
    \end{tikzpicture}
    &
    \begin{tikzpicture}[scale=0.65]
    \path[-, line width=0.4mm]
    (0,1.5) edge (1.17, 0.94)
    (1.17, 0.94) edge (1.46,-0.33)
    (1.46,-0.33) edge (0.65,-1.35)
    (0.65,-1.35) edge (-0.65,-1.35)
    (-0.65,-1.35) edge (-1.46,-0.33)
    (-1.46,-0.33) edge (-1.17,0.94)
    (-1.17,0.94) edge (0,1.5)
    (0,3) edge (2.34, 1.88)
    (2.34, 1.88) edge (2.92,-0.66)
    (2.92,-0.66) edge (1.3,-2.7)
    (1.3,-2.7) edge (-1.3,-2.7)
    (-1.3,-2.7) edge (-2.92,-0.66)
    (-2.92,-0.66) edge (-2.34,1.88)
    (-2.34,1.88) edge (0,3)
    (0,1.5) edge (0,3)
    (1.17, 0.94) edge (2.34, 1.88)
    (1.46,-0.33) edge (2.92,-0.66)
    (0.65,-1.35) edge (1.3,-2.7)
    (-0.65,-1.35) edge (-1.3,-2.7)
    (-1.46,-0.33) edge (-2.92,-0.66)
    (-1.17,0.94) edge (-2.34,1.88)
    (0,1.5) edge (0,0)
    (1.17, 0.94) edge (0,0)
    (1.46,-0.33) edge (0,0)
    (0.65,-1.35) edge (0,0)
    (-0.65,-1.35) edge (0,0)
    (-1.46,-0.33) edge (0,0)
    (-1.17,0.94) edge (0,0);
    \draw [line width=0.25mm, fill=black] (0,1.5) circle (1.5mm);
    \draw [line width=0.25mm, fill=black] (1.17, 0.94) circle (1.5mm);
    \draw [line width=0.25mm, fill=black] (1.46,-0.33) circle (1.5mm);
    \draw [line width=0.25mm, fill=black] (0.65,-1.35) circle (1.5mm);
    \draw [line width=0.25mm, fill=black] (-0.65,-1.35) circle (1.5mm);
    \draw [line width=0.25mm, fill=black] (-1.46,-0.33) circle (1.5mm);
    \draw [line width=0.25mm, fill=black] (-1.17,0.94) circle (1.5mm);
    \draw [line width=0.25mm, fill=white] (0,3) circle (1.5mm);
    \draw [line width=0.25mm, fill=white] (2.34, 1.88) circle (1.5mm);
    \draw [line width=0.25mm, fill=white] (2.92,-0.66) circle (1.5mm);
    \draw [line width=0.25mm, fill=white] (1.3,-2.7) circle (1.5mm);
    \draw [line width=0.25mm, fill=white] (-1.3,-2.7) circle (1.5mm);
    \draw [line width=0.25mm, fill=white] (-2.92,-0.66) circle (1.5mm);
    \draw [line width=0.25mm, fill=white] (-2.34,1.88) circle (1.5mm);
    \draw [line width=0.25mm, fill=white] (0,0) circle (1.5mm);
    \end{tikzpicture}
    &
    \begin{tikzpicture}[scale=0.65]
    \path[-, line width=0.4mm]
    (0.57, 1.39) edge (1.39, 0.57)
    (1.39, 0.57) edge (1.39, -0.57)
    (1.39, -0.57) edge (0.57,-1.39)
    (0.57,-1.39) edge (-0.57,-1.39)
    (-0.57,-1.39) edge (-1.39, -0.57)
    (-1.39, -0.57) edge (-1.39, 0.57)
    (-1.39, 0.57) edge (-0.57, 1.39)
    (-0.57, 1.39) edge (0.57, 1.39)
    (1.14,2.78) edge (2.78, 1.14)
    (2.78, 1.14) edge (2.78, -1.14)
    (2.78, -1.14) edge (1.14,-2.78)
    (1.14,-2.78) edge (-1.14,-2.78)
    (-1.14,-2.78) edge (-2.78, -1.14)
    (-2.78, -1.14) edge (-2.78, 1.14)
    (-2.78, 1.14) edge (-1.14,2.78)
    (-1.14, 2.78) edge (1.14,2.78)
    (1.14,2.78) edge (0.57, 1.39)
    (2.78, 1.14) edge (1.39, 0.57)
    (2.78, -1.14) edge (1.39, -0.57)
    (1.14,-2.78) edge (0.57, -1.39)
    (-1.14,-2.78) edge (-0.57, -1.39)
    (-2.78, -1.14) edge (-1.39, -0.57)
    (-2.78, 1.14) edge (-1.39, 0.57)
    (-1.14, 2.78) edge (-0.57, 1.39)
    (1.71,4.17) edge (1.14,2.78)
    (4.17, 1.71) edge (2.78,1.14)
    (-4.17, 1.71) edge (-2.78,1.14)
    (-1.71,4.17) edge (-1.14,2.78)
    (0,0) edge (0.57, 1.39)
    (0,0) edge (1.39, 0.57)
    (0,0) edge (1.39, -0.57)
    (0,0) edge (0.57, -1.39)
    (0,0) edge (-0.57, -1.39)
    (0,0) edge (-1.39, -0.57)
    (0,0) edge (-1.39, 0.57)
    (0,0) edge (-0.57, 1.39);
    \draw [line width=0.25mm, fill=black] (0.57,1.39) circle (1.5mm);
    \draw [line width=0.25mm, fill=black] (1.39, 0.57) circle (1.5mm);
    \draw [line width=0.25mm, fill=black] (1.39, -0.57) circle (1.5mm);
    \draw [line width=0.25mm, fill=black] (0.57,-1.39) circle (1.5mm);
    \draw [line width=0.25mm, fill=black] (-0.57,-1.39) circle (1.5mm);
    \draw [line width=0.25mm, fill=black] (-1.39, -0.57) circle (1.5mm);
    \draw [line width=0.25mm, fill=black] (-1.39, 0.57) circle (1.5mm);
    \draw [line width=0.25mm, fill=black] (-0.57,1.39) circle (1.5mm);
    \draw [line width=0.25mm, fill=black] (1.14,2.78) circle (1.5mm);
    \draw [line width=0.25mm, fill=black] (2.78, 1.14) circle (1.5mm);
    \draw [line width=0.25mm, fill=white] (2.78, -1.14) circle (1.5mm);
    \draw [line width=0.25mm, fill=white] (1.14,-2.78) circle (1.5mm);
    \draw [line width=0.25mm, fill=white] (-1.14,-2.78) circle (1.5mm);
    \draw [line width=0.25mm, fill=white] (-2.78, -1.14) circle (1.5mm);
    \draw [line width=0.25mm, fill=black] (-2.78, 1.14) circle (1.5mm);
    \draw [line width=0.25mm, fill=black] (-1.14,2.78) circle (1.5mm);
    \draw [line width=0.25mm, fill=white] (1.71,4.17) circle (1.5mm);
    \draw [line width=0.25mm, fill=white] (4.17, 1.71) circle (1.5mm);
    \draw [line width=0.25mm, fill=white] (-4.17, 1.71) circle (1.5mm);
    \draw [line width=0.25mm, fill=white] (-1.71,4.17) circle (1.5mm);
    \draw [line width=0.25mm, fill=white] (0,0) circle (1.5mm);
    \end{tikzpicture}
    \\
    {\Large$\Xi_6$} &
    {\Large$\Xi_7$} &
    {\Large$\Xi_8$} \\
\end{tabular}
\caption{The first six spider graphs, $\Xi_n$ for $3\leq n\leq 8$.}
\label{fig:spider-graphs}
\end{figure}

\subsection{Conditional Local Moves}
In addition to the 7 local moves described in section 3, there are also \emph{conditional local moves} we can perform on rnpds.  These are similar to the local moves in that they do not change the response matrix, but the formulas are not subtraction free.  Thus, we can only perform the conditional local moves as long as they do not give us any negative conductances.\\

We have the following conditional local moves (see Appendix~\ref{app:moves-formulas} for formulas):
\begin{itemize}
    \item Triangle Conditional Local Move
    
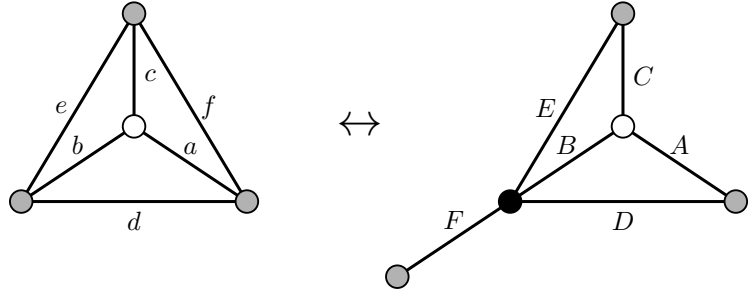
\begin{center}
    \begin{tikzpicture}
    \path[-, line width=0.4mm]
    (0,0) edge node[above]{$b$} (1.5,1)
    (1.5,2.35) edge node[right]{$c$} (1.5,1)
    (3,0) edge node[above]{$a$} (1.5,1)
    (0,0) edge node[left]{$e$} (1.5,2.5)
    (1.5,2.5) edge node[right]{$f$} (3,0)
    (3,0) edge node[below]{$d$} (0,0);
    \draw [line width=0.25mm, fill=gray!60] (0,0) circle (1.5mm);
    \draw [line width=0.25mm, fill=gray!60] (1.5,2.5) circle (1.5mm);
    \draw [line width=0.25mm, fill=gray!60] (3,0) circle (1.5mm);
    \draw [line width=0.25mm, fill=white] (1.5,1) circle (1.5mm);
    \node at (-0.5,1) {\phantom{1}};
    \node at (4.5,1) {\LARGE$\leftrightarrow$};
    \path[-, line width=0.4mm]
    (6.5,0) edge node[above]{$B$} (8,1)
    (8,2.35) edge node[right]{$C$} (8,1)
    (9.5,0) edge node[above]{$A$} (8,1)
    (6.5,0) edge node[left]{$E$} (8,2.5)
    (5,-1) edge node[above]{$F$} (6.5,0)
    (9.5,0) edge node[below]{$D$} (6.5,0);
    \draw [line width=0.25mm, fill=gray!60] (5,-1) circle (1.5mm);
    \draw [line width=0.25mm, fill=black] (6.5,0) circle (1.5mm);
    \draw [line width=0.25mm, fill=gray!60] (8,2.5) circle (1.5mm);
    \draw [line width=0.25mm, fill=gray!60] (9.5,0) circle (1.5mm);
    \draw [line width=0.25mm, fill=white] (8,1) circle (1.5mm);
    \end{tikzpicture}
    \end{center}
    \item Square Conditional Local Move
    
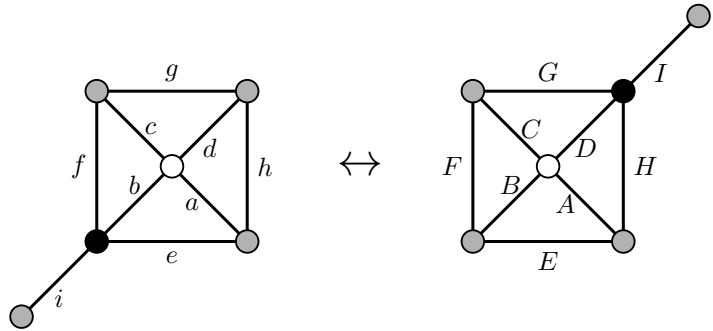
\begin{center}
    \begin{tikzpicture}
    \path[-, line width=0.4mm]
    (0,0) edge node[above]{$b$} (1,1)
    (0,2) edge node[right]{$c$} (1,1)
    (2,0) edge node[left]{$a$} (1,1)
    (0,0) edge node[below]{$e$} (2,0)
    (0,0) edge node[left]{$f$} (0,2)
    (2,2) edge node[below]{$d$} (1,1)
    (0,2) edge node[above]{$g$} (2,2)
    (2,2) edge node[right]{$h$} (2,0)
    (-1,-1) edge node[below]{$i$} (0,0);
    \draw [line width=0.25mm, fill=black] (0,0) circle (1.5mm);
    \draw [line width=0.25mm, fill=gray!60] (0,2) circle (1.5mm);
    \draw [line width=0.25mm, fill=gray!60] (2,0) circle (1.5mm);
    \draw [line width=0.25mm, fill=gray!60] (2,2) circle (1.5mm);
    \draw [line width=0.25mm, fill=white] (1,1) circle (1.5mm);
    \draw [line width=0.25mm, fill=gray!60] (-1,-1) circle (1.5mm);
    \node at (-0.5,1) {\phantom{1}};
    \node at (3.5,1) {\LARGE$\leftrightarrow$};
    \path[-, line width=0.4mm]
    (5,0) edge node[above]{$B$} (6,1)
    (5,2) edge node[right]{$C$} (6,1)
    (7,0) edge node[left]{$A$} (6,1)
    (5,0) edge node[below]{$E$} (7,0)
    (5,0) edge node[left]{$F$} (5,2)
    (7,2) edge node[below]{$D$} (6,1)
    (5,2) edge node[above]{$G$} (7,2)
    (7,2) edge node[right]{$H$} (7,0)
    (8,3) edge node[below]{$I$} (7,2);
    \draw [line width=0.25mm, fill=gray!60] (5,0) circle (1.5mm);
    \draw [line width=0.25mm, fill=gray!60] (5,2) circle (1.5mm);
    \draw [line width=0.25mm, fill=gray!60] (7,0) circle (1.5mm);
    \draw [line width=0.25mm, fill=black] (7,2) circle (1.5mm);
    \draw [line width=0.25mm, fill=white] (6,1) circle (1.5mm);
    \draw [line width=0.25mm, fill=gray!60] (8,3) circle (1.5mm);
    \end{tikzpicture}
    \end{center}
    \item Pentagon Conditional Local Move
    
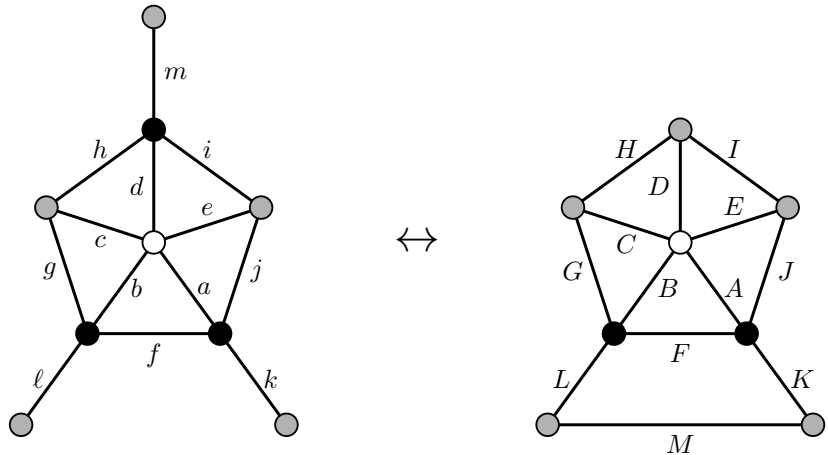
\begin{center}
    \begin{tikzpicture}
    \path[-, line width=0.4mm]
    (0,0) edge node[right]{$a$} (0.882,-1.213)
    (0,0) edge node[right]{$b$} (-0.882,-1.213)
    (0,0) edge node[below]{$c$} (-1.427,0.464)
    (0,0) edge node[left]{$d$} (0,1.5)
    (0,0) edge node[above]{$e$} (1.427,0.464)
    (0.882,-1.213) edge node[below]{$f$} (-0.882,-1.213)
    (-0.882,-1.213) edge node[left]{$g$} (-1.427,0.464)
    (-1.427,0.464) edge node[above]{$h$} (0,1.5)
    (0,1.5) edge node[above]{$i$} (1.427,0.464)
    (1.427,0.464) edge node[right]{$j$} (0.882,-1.213)
    (1.764,-2.426) edge node[right]{$k$} (0.882,-1.213)
    (-1.764,-2.426) edge node[left]{$\ell$} (-0.882,-1.213)
    (0,3) edge node[right]{$m$} (0,1.5);
    \draw [line width=0.25mm, fill=black] (0,1.5) circle (1.5mm);
    \draw [line width=0.25mm, fill=black] (0.882,-1.213) circle (1.5mm);
    \draw [line width=0.25mm, fill=black] (-0.882,-1.213) circle (1.5mm);
    \draw [line width=0.25mm, fill=gray!60] (1.427,0.464) circle (1.5mm);
    \draw [line width=0.25mm, fill=gray!60] (-1.427,0.464) circle (1.5mm);
    \draw [line width=0.25mm, fill=white] (0,0) circle (1.5mm);
    \draw [line width=0.25mm, fill=gray!60] (0,3) circle (1.5mm);
    \draw [line width=0.25mm, fill=gray!60] (1.764,-2.426) circle (1.5mm);
    \draw [line width=0.25mm, fill=gray!60] (-1.764,-2.426) circle (1.5mm);
    \node at (3.5,0) {\LARGE$\leftrightarrow$};
    \path[-, line width=0.4mm]
    (7,0) edge node[right]{$A$} (7.882,-1.213)
    (7,0) edge node[right]{$B$} (6.118,-1.213)
    (7,0) edge node[below]{$C$} (5.573,0.464)
    (7,0) edge node[left]{$D$} (7,1.5)
    (7,0) edge node[above]{$E$} (8.427,0.464)
    (7.882,-1.213) edge node[below]{$F$} (6.118,-1.213)
    (6.118,-1.213) edge node[left]{$G$} (5.573,0.464)
    (5.573,0.464) edge node[above]{$H$} (7,1.5)
    (7,1.5) edge node[above]{$I$} (8.427,0.464)
    (8.427,0.464) edge node[right]{$J$} (7.882,-1.213)
    (8.764,-2.426) edge node[right]{$K$} (7.882,-1.213)
    (5.236,-2.426) edge node[left]{$L$} (6.118,-1.213)
    (8.764,-2.426) edge node[below]{$M$} (5.236,-2.426);
    \draw [line width=0.25mm, fill=gray!60] (7,1.5) circle (1.5mm);
    \draw [line width=0.25mm, fill=black] (7.882,-1.213) circle (1.5mm);
    \draw [line width=0.25mm, fill=black] (6.118,-1.213) circle (1.5mm);
    \draw [line width=0.25mm, fill=gray!60] (8.427,0.464) circle (1.5mm);
    \draw [line width=0.25mm, fill=gray!60] (5.573,0.464) circle (1.5mm);
    \draw [line width=0.25mm, fill=white] (7,0) circle (1.5mm);
    \draw [line width=0.25mm, fill=gray!60] (8.764,-2.426) circle (1.5mm);
    \draw [line width=0.25mm, fill=gray!60] (5.236,-2.426) circle (1.5mm);
    \end{tikzpicture}
    \end{center}
\end{itemize}

\newpage
\begin{conj}
We conjecture a hexagon conditional local move as well:
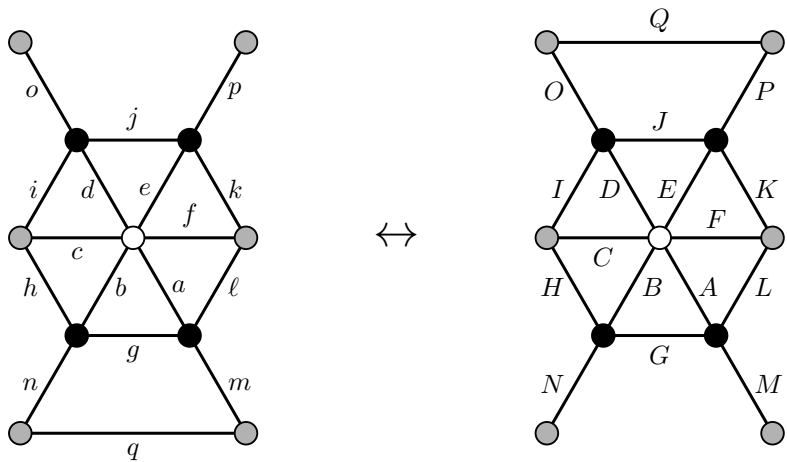
\begin{center}
    \begin{tikzpicture}
    \path[-, line width=0.4mm]
    (0,0) edge node[right]{$a$} (0.75,-1.3)
    (0,0) edge node[right]{$b$} (-0.75,-1.3)
    (0,0) edge node[below]{$c$} (-1.5,0)
    (0,0) edge node[left]{$d$} (-0.75,1.3)
    (0,0) edge node[left]{$e$} (0.75,1.3)
    (0,0) edge node[above]{$f$} (1.5,0)
    (0.75,-1.3) edge node[below]{$g$} (-0.75,-1.3)
    (-0.75,-1.3) edge node[left]{$h$} (-1.5,0)
    (-1.5,0) edge node[left]{$i$} (-0.75,1.3)
    (-0.75,1.3) edge node[above]{$j$} (0.75,1.3)
    (0.75,1.3) edge node[right]{$k$} (1.5,0)
    (1.5,0) edge node[right]{$\ell$} (0.75,-1.3)
    (0.75,-1.3) edge node[right]{$m$} (1.5,-2.6)
    (-0.75,-1.3) edge node[left]{$n$} (-1.5,-2.6)
    (-0.75,1.3) edge node[left]{$o$} (-1.5,2.6)
    (0.75,1.3) edge node[right]{$p$} (1.5,2.6)
    (-1.5,-2.6) edge node[below]{$q$} (1.5,-2.6);
    \draw [line width=0.25mm, fill=black] (0.75,1.3) circle (1.5mm);
    \draw [line width=0.25mm, fill=black] (0.75,-1.3) circle (1.5mm);
    \draw [line width=0.25mm, fill=black] (-0.75,1.3) circle (1.5mm);
    \draw [line width=0.25mm, fill=black] (-0.75,-1.3) circle (1.5mm);
    \draw [line width=0.25mm, fill=gray!60] (1.5,2.6) circle (1.5mm);
    \draw [line width=0.25mm, fill=gray!60] (1.5,-2.6) circle (1.5mm);
    \draw [line width=0.25mm, fill=gray!60] (-1.5,2.6) circle (1.5mm);
    \draw [line width=0.25mm, fill=gray!60] (-1.5,-2.6) circle (1.5mm);
    \draw [line width=0.25mm, fill=gray!60] (-1.5,0) circle (1.5mm);
    \draw [line width=0.25mm, fill=white] (0,0) circle (1.5mm);
    \draw [line width=0.25mm, fill=gray!60] (1.5,0) circle (1.5mm);
    \node at (3.5,0) {\LARGE$\leftrightarrow$};
    \path[-, line width=0.4mm]
    (7,0) edge node[right]{$A$} (7.75,-1.3)
    (7,0) edge node[right]{$B$} (6.25,-1.3)
    (7,0) edge node[below]{$C$} (5.5,0)
    (7,0) edge node[left]{$D$} (6.25,1.3)
    (7,0) edge node[left]{$E$} (7.75,1.3)
    (7,0) edge node[above]{$F$} (8.5,0)
    (7.75,-1.3) edge node[below]{$G$} (6.25,-1.3)
    (6.25,-1.3) edge node[left]{$H$} (5.5,0)
    (5.5,0) edge node[left]{$I$} (6.25,1.3)
    (6.25,1.3) edge node[above]{$J$} (7.75,1.3)
    (7.75,1.3) edge node[right]{$K$} (8.5,0)
    (8.5,0) edge node[right]{$L$} (7.75,-1.3)
    (7.75,-1.3) edge node[right]{$M$} (8.5,-2.6)
    (6.25,-1.3) edge node[left]{$N$} (5.5,-2.6)
    (6.25,1.3) edge node[left]{$O$} (5.5,2.6)
    (7.75,1.3) edge node[right]{$P$} (8.5,2.6)
    (5.5,2.6) edge node[above]{$Q$} (8.5,2.6);
    \draw [line width=0.25mm, fill=black] (7.75,1.3) circle (1.5mm);
    \draw [line width=0.25mm, fill=black] (7.75,-1.3) circle (1.5mm);
    \draw [line width=0.25mm, fill=black] (6.25,1.3) circle (1.5mm);
    \draw [line width=0.25mm, fill=black] (6.25,-1.3) circle (1.5mm);
    \draw [line width=0.25mm, fill=gray!60] (8.5,2.6) circle (1.5mm);
    \draw [line width=0.25mm, fill=gray!60] (8.5,-2.6) circle (1.5mm);
    \draw [line width=0.25mm, fill=gray!60] (5.5,2.6) circle (1.5mm);
    \draw [line width=0.25mm, fill=gray!60] (5.5,-2.6) circle (1.5mm);
    \draw [line width=0.25mm, fill=gray!60] (5.5,0) circle (1.5mm);
    \draw [line width=0.25mm, fill=white] (7,0) circle (1.5mm);
    \draw [line width=0.25mm, fill=gray!60] (8.5,0) circle (1.5mm);
    \end{tikzpicture}
    \end{center}

We have conjectural formulas for this move (see Appendix~\ref{app:moves-formulas}).  We were able to check that these formulas preserve every entry of the response matrix except for the entry in the row and column corresponding to the interior boundary vertex.  The computations were too large to check for this entry.  However, we did check that the response matrices were the same for several choices of values for the variables in one of the networks.
\end{conj}

\begin{conj}
Notice that each of the conditional local moves listed above is a move on a subgraph of $\Pi_n$.  We conjecture that there exists a conditional local move for each $n\geq 3$ where applying the local move on $\Pi_n$ does the following:
\begin{itemize}
    \item For $n\equiv 3\mod 4$, we change a boundary edge from the outermost layer into a boundary spike on the opposite side of the graph.
    \item For $n\equiv 0\mod 4$, we change a boundary spike from the outermost partial layer into a boundary spike on the opposite side of the graph.
    \item For $n\equiv 1\mod 4$, we change a boundary spike from the outermost partial layer into a boundary edge on the opposite side of the graph.
    \item For $n\equiv 2\mod 4$, we change a boundary edge from the outermost partial layer into a boundary edge on the opposite side of the graph.
\end{itemize}
It also seems that the formulas for these moves have a combinatorial interpretation in terms of groves or a grove-like structure (see~\cite{2011arXiv1104.4998L}).  For example the numerator in the formula for $G, H$, and $I$ in the pentagon conditional local move is a sum of groves.
\end{conj}

\section{Recoverability Conditions}
\subsection{Recoverable Rnpds from Critical Cprns}

\begin{lemma}[Lemma 11.3 from \cite{curtis_ingerman_morrow}] \label{alwaysedgespike}
Let $\Gamma$ be a critical cprn. Then $\Gamma$ has a boundary edge or a boundary spike.
\end{lemma}

\begin{lemma}[Lemma 11.1 from \cite{curtis_ingerman_morrow}]
\label{stillcriticaledge}
Let $\Gamma$ be a critical cprn. Then $\Gamma$ remains critical after deleting a boundary edge.
\end{lemma}
\begin{lemma}[Lemma 11.2 from \cite{curtis_ingerman_morrow}]
\label{stillcriticalspike}
Let $\Gamma$ be a critical cprn. Then $\Gamma$ remains critical after contracting a boundary spike.
\end{lemma}

\begin{lemma} \label{weknowconductance}
Let $\Gamma$ be an rnpd with interior boundary vertex $b$, and let $e$ be a boundary edge or a boundary spike. If deleting or contracting $e = pq$ breaks some connection $(P, Q)$ s.t. $b \notin P \cup Q$, then we can derive the conductance of $e$ from $\Lambda(\Gamma)$.
\end{lemma}
\begin{proof}
Corollaries 4.3 and 4.4 of \cite{curtis_ingerman_morrow} prove this for the cprn case. However, this remains true for rndps if we restrict $P$ and $Q$ to not include the interior boundary vertex.
%Let $\Gamma'$ be the cprn that results from deleting $b$ and all edges incident to $b$. Recall that paths in a connection cannot include boundary vertices. Thus, restricting $P$ and $Q$ to not include $b$ is equivalent to saying that deleting or contracting $e$ breaks connection $(P, Q)$ in $\Gamma'$. By corollaries 4.3 and 4.4 in  \cite{curtis_ingerman_morrow}, 
\end{proof}

\begin{defn}
A \textit{star graph} is a central boundary vertex with only boundary spikes or pendants attached. \textit{Inserting a star graph into a face of a graph} means adding a boundary vertex inside the face and connecting it to any number of vertices on the face.
\end{defn}

\begin{ex}
Consider the following graph:
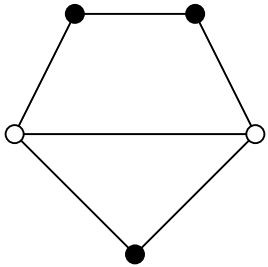
\begin{center}
\begin{tikzpicture}[scale=0.8]
\draw [line width = 0.25mm, fill = black] (1, 2)--(2,0);
\draw [line width = 0.25mm, fill = black] (-1, 2)--(-2,0);
\draw [line width = 0.25mm, fill = black] (-1, 2)--(1,2);
\draw [line width = 0.25mm, fill = black] (0, -2)--(2,0);
\draw [line width = 0.25mm, fill = black] (0, -2)--(-2,0);
\draw [line width = 0.25mm, fill = black]  (2,0)--(-2,0) ;
\draw [line width=0.25mm, fill=black] (1, 2) circle (1.5mm);
\draw [line width=0.25mm, fill=black] (-1, 2) circle (1.5mm);
\draw [line width=0.25mm, fill=black] (0, -2) circle (1.5mm);
\draw [line width=0.25mm, fill=white] (-2,0) circle (1.5mm);
\draw [line width=0.25mm, fill=white] (2,0) circle (1.5mm);
\end{tikzpicture}
\end{center}

One way to insert a star graph into the upper face is illustrated below:
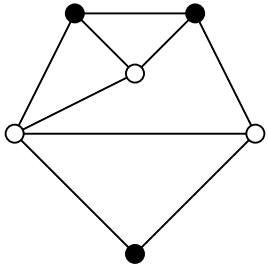
\begin{center}
\begin{tikzpicture}[scale=0.8]
\draw [line width = 0.25mm, fill = black] (0,1)--(-2,0);
\draw [line width = 0.25mm, fill = black] (0,1)--(-1,2);
\draw [line width = 0.25mm, fill = black] (0,1)--(1,2);
\draw [line width = 0.25mm, fill = black] (1, 2)--(2,0);
\draw [line width = 0.25mm, fill = black] (-1, 2)--(-2,0);
\draw [line width = 0.25mm, fill = black] (-1, 2)--(1,2);
\draw [line width = 0.25mm, fill = black] (0, -2)--(2,0);
\draw [line width = 0.25mm, fill = black] (0, -2)--(-2,0);
\draw [line width = 0.25mm, fill = black]  (2,0)--(-2,0) ;
\draw [line width=0.25mm, fill=black] (1, 2) circle (1.5mm);
\draw [line width=0.25mm, fill=black] (-1, 2) circle (1.5mm);
\draw [line width=0.25mm, fill=black] (0, -2) circle (1.5mm);
\draw [line width=0.25mm, fill=white] (-2,0) circle (1.5mm);
\draw [line width=0.25mm, fill=white] (2,0) circle (1.5mm);
\draw [line width=0.25mm, fill=white] (0,1) circle (1.5mm);
\end{tikzpicture}
\end{center}
\end{ex}

\begin{thm}
Assume $\Gamma$ is a critical cprn. If $\Gamma'$ be the rnpd obtained from inserting a star graph into any face of $\Gamma$, then $\Gamma'$ is recoverable.
\end{thm}
\begin{proof}
Let $\Gamma' = (V, E)$ and $S=(V_S,E_S)$ be the star graph that was inserted into $\Gamma$ to obtain $\Gamma'$.
%and the center face of the corresponding $\Gamma$ containing the interior boundary vertex be $f$. Let $G = (V', E')$ be the interior graph of $f$, i.e. the star graph with the middle boundary vertex $b$.
Assume we are given $\Lambda(\Gamma')$. First, note that $(V \setminus\{b\}, E \setminus E_S) = \Gamma$, which is known to be critical. So, by Lemma~\ref{alwaysedgespike}, we always have a boundary edge or boundary spike with which to begin the following process. 
First, pick a boundary edge or boundary spike, $e$, to delete or contract. Notice that deleting or contracting $e$ must break a connection $(P, Q)$ in $\Gamma$ by the definition of critical. Since paths in connections cannot pass through other boundary vertices, adding back in $S$ does not repair the broken connection, and $(P, Q)$ is broken in $\Gamma'$. Thus, by Lemma~\ref{weknowconductance}, we know the conductance of $e$. Knowing this allows us to derive the response matrix of the graph after contracting or deleting $e$, which we will call $\Gamma''$ (see section 8 of \cite{curtis_ingerman_morrow}). Now, by Lemmas~\ref{stillcriticaledge} and~\ref{stillcriticalspike}, our $\Gamma''$ without $S$ is still critical. Then, we can continue this process, deleting and contracting boundary edges and spikes one by one, until we are left with $S$. In other words, we will eventually know the conductance of every edge in $\Gamma'$ except for those in $G$, but will know the response matrix of $G$. However, at that point, $S$ will be a cprn with only boundary vertices, and hence its response matrix is exactly its Kirchhoff matrix. Therefore, $\Gamma'$ is recoverable.
\end{proof}

\begin{ex}
Note the $\Xi_n$ is $\Pi_n$ with a star graph inserted into a face.  So, all the spider graphs are recoverable.
\end{ex}

\begin{ex}
By Theorem~\ref{thm:same-z-seq}, all irreducible graphs that have the same z-sequence as a spider graph are recoverable.
\end{ex}

\subsection{Necessary Condition for Recoverablilty of Rnpds}
Before stating the necessary condition, we first need to define the following algorithm for obtaining a cprn from a certain type of rnpd. Suppose we're given an rnpd with the interior boundary vertex $b$.\\

\textbf{Algorithm 1.} 
\begin{enumerate}
    \item Turn all the neighbors of $b$ into boundary vertices.
    \item Remove $b$ from the graph (with all the edges incident to it).
    \item If the resulting graph is a cprn, stop.
    \item Otherwise, apply the above steps to each boundary vertex (in any order) that can't be placed on the disk until it results in a cprn. 
\end{enumerate}

Recall that in any rnpd, the interior boundary vertex has to be inside a polygon (one of the faces of a cprn).

\begin{thm}
Assume $\Gamma'$ is an rnpd where the internal boundary vertex is the center of a star graph in the interior of some face $f$ of the cprn $\Gamma$. Apply Algorithm 1 iteratively starting at the internal boundary vertex, and call the resulting cprn $\Gamma''$. Then $\Gamma''$ is critical if $\Gamma'$ is recoverable.
\end{thm}
\begin{proof}
Assume $\Gamma''$ is not critical. Then the number of edges can be reduced. Note that $\Gamma''$ is isomorphic to a subgraph of $\Gamma'$, but not necessarily to a subnetwork (not respecting the type of vertices). So $\Gamma''$ may have more boundary vertices. Having more boundary vertices may only decrease the number of possible moves. This implies that the number of edges in $\Gamma'$ could be reduced to begin with, and thus $\Gamma'$ is not recoverable.
\end{proof}

\section*{Acknowledgements}
This research was carried out as part of the 2018 REU program at the School of Mathematics at University
of Minnesota, Twin Cities. The authors are grateful for the support
of NSF RTG grant DMS-1745638. They would also like to thank Sunita Chepuri and Pavlo Pylyavskyy for their guidance and
encouragement, Eric Stucky for his many helpful comments on the manuscript, and Alexandra Embry, Sylvia Frank, and John O'Brien for their insights.

\appendix
\titleformat{\section}{\Large\bfseries}{Appendix \thesection}{1em}{}
\section{Formulas for Conditional Local Moves}\label{app:moves-formulas}

Here we list the conditional local moves with formulas.

\begin{itemize}
    \item Triangle Conditional Local Move
    \begin{center}
    \begin{tikzpicture}
    \path[-, line width=0.4mm]
    (0,0) edge node[above]{$b$} (1.5,1)
    (1.5,2.35) edge node[right]{$c$} (1.5,1)
    (3,0) edge node[above]{$a$} (1.5,1)
    (0,0) edge node[left]{$e$} (1.5,2.5)
    (1.5,2.5) edge node[right]{$f$} (3,0)
    (3,0) edge node[below]{$d$} (0,0);
    \draw [line width=0.25mm, fill=gray!60] (0,0) circle (1.5mm);
    \draw [line width=0.25mm, fill=gray!60] (1.5,2.5) circle (1.5mm);
    \draw [line width=0.25mm, fill=gray!60] (3,0) circle (1.5mm);
    \draw [line width=0.25mm, fill=white] (1.5,1) circle (1.5mm);
    \node at (-0.5,1) {\phantom{1}};
    \node at (4.5,1) {\LARGE$\leftrightarrow$};
    \path[-, line width=0.4mm]
    (6.5,0) edge node[above]{$B$} (8,1)
    (8,2.35) edge node[right]{$C$} (8,1)
    (9.5,0) edge node[above]{$A$} (8,1)
    (6.5,0) edge node[left]{$E$} (8,2.5)
    (5,-1) edge node[above]{$F$} (6.5,0)
    (9.5,0) edge node[below]{$D$} (6.5,0);
    \draw [line width=0.25mm, fill=gray!60] (5,-1) circle (1.5mm);
    \draw [line width=0.25mm, fill=black] (6.5,0) circle (1.5mm);
    \draw [line width=0.25mm, fill=gray!60] (8,2.5) circle (1.5mm);
    \draw [line width=0.25mm, fill=gray!60] (9.5,0) circle (1.5mm);
    \draw [line width=0.25mm, fill=white] (8,1) circle (1.5mm);
    \end{tikzpicture}
    \end{center}
    \begin{tabular}{l @{\hspace{0.75in}}l}
    $\displaystyle a=\frac{AB+AD+AE+AF+BD}{B+D+E+F}$ & $\displaystyle A=\frac{ea-bf}{e}$\\
    $\displaystyle b=\frac{BF}{B+D+E+F}$ & $\displaystyle B=\frac{b(bf+de+df+ef)}{de}$\\
    $\displaystyle c=\frac{BC+BE+CD+CE+CF}{B+D+E+F}$ & $\displaystyle C=\frac{cd-bf}{d}$\\
    $\displaystyle d=\frac{DF}{B+D+E+F}$ & $\displaystyle D=\frac{bf+de+df+fe}{e}$\\
    $\displaystyle e=\frac{EF}{B+D+E+F}$ & $\displaystyle E=\frac{bf+de+df+fe}{d}$\\
    $\displaystyle f=\frac{DE}{B+D+E+F}$ & $\displaystyle F=\frac{bf+de+df+fe}{f}$
    \end{tabular}
    \item Square Conditional Local Move
    
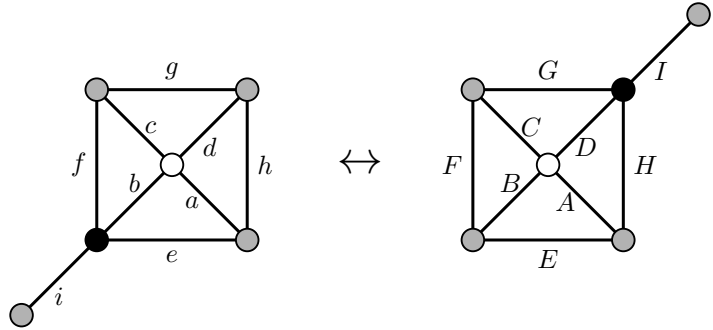
\begin{center}
    \begin{tikzpicture}
    \path[-, line width=0.4mm]
    (0,0) edge node[above]{$b$} (1,1)
    (0,2) edge node[right]{$c$} (1,1)
    (2,0) edge node[left]{$a$} (1,1)
    (0,0) edge node[below]{$e$} (2,0)
    (0,0) edge node[left]{$f$} (0,2)
    (2,2) edge node[below]{$d$} (1,1)
    (0,2) edge node[above]{$g$} (2,2)
    (2,2) edge node[right]{$h$} (2,0)
    (-1,-1) edge node[below]{$i$} (0,0);
    \draw [line width=0.25mm, fill=black] (0,0) circle (1.5mm);
    \draw [line width=0.25mm, fill=gray!60] (0,2) circle (1.5mm);
    \draw [line width=0.25mm, fill=gray!60] (2,0) circle (1.5mm);
    \draw [line width=0.25mm, fill=gray!60] (2,2) circle (1.5mm);
    \draw [line width=0.25mm, fill=white] (1,1) circle (1.5mm);
    \draw [line width=0.25mm, fill=gray!60] (-1,-1) circle (1.5mm);
    \node at (-0.5,1) {\phantom{1}};
    \node at (3.5,1) {\LARGE$\leftrightarrow$};
    \path[-, line width=0.4mm]
    (5,0) edge node[above]{$B$} (6,1)
    (5,2) edge node[right]{$C$} (6,1)
    (7,0) edge node[left]{$A$} (6,1)
    (5,0) edge node[below]{$E$} (7,0)
    (5,0) edge node[left]{$F$} (5,2)
    (7,2) edge node[below]{$D$} (6,1)
    (5,2) edge node[above]{$G$} (7,2)
    (7,2) edge node[right]{$H$} (7,0)
    (8,3) edge node[below]{$I$} (7,2);
    \draw [line width=0.25mm, fill=gray!60] (5,0) circle (1.5mm);
    \draw [line width=0.25mm, fill=gray!60] (5,2) circle (1.5mm);
    \draw [line width=0.25mm, fill=gray!60] (7,0) circle (1.5mm);
    \draw [line width=0.25mm, fill=black] (7,2) circle (1.5mm);
    \draw [line width=0.25mm, fill=white] (6,1) circle (1.5mm);
    \draw [line width=0.25mm, fill=gray!60] (8,3) circle (1.5mm);
    \end{tikzpicture}
    \end{center}
    In this case the move is symmetric, so we only include formulas for one direction.
    
    \vspace{0.05in}
    $\displaystyle A=\frac{abg+aeg+afg+agi+beg-def}{g(b+e+f+i)}$
    
    \vspace{0.05in}
    $\displaystyle B=\frac{bi}{b+e+f+i}$
    
    \vspace{0.05in}
    $\displaystyle C=\frac{bch+bfh+ceh+cfh+chi-def}{h(b+e+f+i)}$
    
    \vspace{0.05in}
    $\displaystyle D=\frac{d(bgh+def+efg+efh+egh+fgh+ghi)}{gh(b+e+f+i)}$
    
    \vspace{0.05in}
    $\displaystyle E=\frac{ei}{b+e+f+i}$
    
    \vspace{0.05in}
    $\displaystyle F=\frac{fi}{b+e+f+i}$
    
    \vspace{0.05in}
    $\displaystyle G=\frac{bgh+def+efg+efh+egh+fgh+ghi}{h(b+e+f+i)}$
    
    \vspace{0.05in}
    $\displaystyle H=\frac{bgh+def+efg+efh+egh+fgh+ghi}{g(b+e+f+i)}$
    
    \vspace{0.05in}
    $\displaystyle I=\frac{bgh+def+efg+efh+egh+fgh+ghi}{ef}$
    
    \item Pentagon Conditional Local Move
    
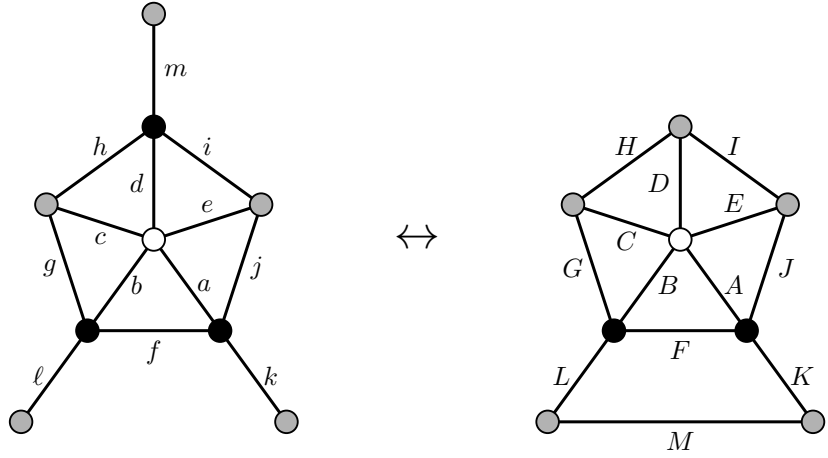
\begin{center}
    \begin{tikzpicture}
    \path[-, line width=0.4mm]
    (0,0) edge node[right]{$a$} (0.882,-1.213)
    (0,0) edge node[right]{$b$} (-0.882,-1.213)
    (0,0) edge node[below]{$c$} (-1.427,0.464)
    (0,0) edge node[left]{$d$} (0,1.5)
    (0,0) edge node[above]{$e$} (1.427,0.464)
    (0.882,-1.213) edge node[below]{$f$} (-0.882,-1.213)
    (-0.882,-1.213) edge node[left]{$g$} (-1.427,0.464)
    (-1.427,0.464) edge node[above]{$h$} (0,1.5)
    (0,1.5) edge node[above]{$i$} (1.427,0.464)
    (1.427,0.464) edge node[right]{$j$} (0.882,-1.213)
    (1.764,-2.426) edge node[right]{$k$} (0.882,-1.213)
    (-1.764,-2.426) edge node[left]{$\ell$} (-0.882,-1.213)
    (0,3) edge node[right]{$m$} (0,1.5);
    \draw [line width=0.25mm, fill=black] (0,1.5) circle (1.5mm);
    \draw [line width=0.25mm, fill=black] (0.882,-1.213) circle (1.5mm);
    \draw [line width=0.25mm, fill=black] (-0.882,-1.213) circle (1.5mm);
    \draw [line width=0.25mm, fill=gray!60] (1.427,0.464) circle (1.5mm);
    \draw [line width=0.25mm, fill=gray!60] (-1.427,0.464) circle (1.5mm);
    \draw [line width=0.25mm, fill=white] (0,0) circle (1.5mm);
    \draw [line width=0.25mm, fill=gray!60] (0,3) circle (1.5mm);
    \draw [line width=0.25mm, fill=gray!60] (1.764,-2.426) circle (1.5mm);
    \draw [line width=0.25mm, fill=gray!60] (-1.764,-2.426) circle (1.5mm);
    \node at (3.5,0) {\LARGE$\leftrightarrow$};
    \path[-, line width=0.4mm]
    (7,0) edge node[right]{$A$} (7.882,-1.213)
    (7,0) edge node[right]{$B$} (6.118,-1.213)
    (7,0) edge node[below]{$C$} (5.573,0.464)
    (7,0) edge node[left]{$D$} (7,1.5)
    (7,0) edge node[above]{$E$} (8.427,0.464)
    (7.882,-1.213) edge node[below]{$F$} (6.118,-1.213)
    (6.118,-1.213) edge node[left]{$G$} (5.573,0.464)
    (5.573,0.464) edge node[above]{$H$} (7,1.5)
    (7,1.5) edge node[above]{$I$} (8.427,0.464)
    (8.427,0.464) edge node[right]{$J$} (7.882,-1.213)
    (8.764,-2.426) edge node[right]{$K$} (7.882,-1.213)
    (5.236,-2.426) edge node[left]{$L$} (6.118,-1.213)
    (8.764,-2.426) edge node[below]{$M$} (5.236,-2.426);
    \draw [line width=0.25mm, fill=gray!60] (7,1.5) circle (1.5mm);
    \draw [line width=0.25mm, fill=black] (7.882,-1.213) circle (1.5mm);
    \draw [line width=0.25mm, fill=black] (6.118,-1.213) circle (1.5mm);
    \draw [line width=0.25mm, fill=gray!60] (8.427,0.464) circle (1.5mm);
    \draw [line width=0.25mm, fill=gray!60] (5.573,0.464) circle (1.5mm);
    \draw [line width=0.25mm, fill=white] (7,0) circle (1.5mm);
    \draw [line width=0.25mm, fill=gray!60] (8.764,-2.426) circle (1.5mm);
    \draw [line width=0.25mm, fill=gray!60] (5.236,-2.426) circle (1.5mm);
    \end{tikzpicture}
    \end{center}
    Let $\alpha=ABM+AFM+AGM+ALM+BFM+BJM+BKM+FGM+FJM+FKL+FKM+FLM+GJM+GKM+JLM+KLM, \beta=FKL+FKM+FLM+KLM, \gamma=fgj(d+h+i+m)$, and $\delta=abhi+afhi+aghi+bfhi+bhij+fghi+fhij+ghij$.
    
    \vspace{0.05in}
    $\displaystyle a=\frac{A(\beta+BKM+GKM)}{\alpha}$
    
    \vspace{0.05in}
    $\displaystyle b=\frac{B(\beta+ALM+JLM)}{\alpha}$
    
    \vspace{0.05in}
    $\displaystyle c=C+\frac{ABGIM+AFGIM+BFGIM+BGIJM+BGIKM-DFGJM}{I\alpha}$
    
    \vspace{0.05in}
    $\displaystyle d=D+\frac{D(DFGJM+HFGJM+IFGJM)}{IH\alpha}$
    
    \vspace{0.05in}
    $\displaystyle e=E+\frac{ABHJM+AFHJM+AGHJM+AHJLM+BFHJM-DFGJM}{H\alpha}$
    
    \vspace{0.05in}
    $\displaystyle f=\frac{(\beta+BKM+GKM)(\beta+ALM+JLM)}{LK\alpha}$
    
    \vspace{0.05in}
    $\displaystyle g=\frac{G(\beta+ALM+JLM)}{\alpha}$
    
    \vspace{0.05in}
    $\displaystyle h=H+\frac{DFGJM+HFGJM+IFGJM}{I\alpha}$
    
    \vspace{0.05in}
    $\displaystyle i=I+\frac{DFGJM+HFGJM+IFGJM}{H\alpha}$
    
    \vspace{0.05in}
    $\displaystyle j=\frac{J(\beta+BKM+GKM)}{\alpha}$
    
    \vspace{0.05in}
    $\displaystyle k=\frac{\beta+BKM+GKM}{FL}$

    \vspace{0.05in}
    $\displaystyle \ell=\frac{\beta+ALM+JLM}{KF}$
    
    \vspace{0.05in}
    $\displaystyle m=\frac{HI\alpha+DFGJM+HFGJM+IFGJM}{FGJM}$
    
    \vspace{0.05in}
    $\displaystyle A=a+\frac{a(\delta+ahi\ell+hij\ell)}{\gamma}$
    
    \vspace{0.05in}
    $\displaystyle B=b+\frac{b(\delta+bhik+ghik)}{\gamma}$
    
    \vspace{0.05in}
    $\displaystyle C=c+\frac{g(dfhj-abhi-afhi-bfhi-bhij-ihbk)}{\gamma}$
    
    \vspace{0.05in}
    $\displaystyle D=\frac{dfgjm}{\gamma}$
    
    \vspace{0.05in}
    $\displaystyle E=e+\frac{j(dfgi-abhi-afhi-aghi-ahi\ell-bfhi)}{\gamma}$
    
    \vspace{0.05in}
    $\displaystyle F=\frac{f(\gamma+\delta+ahi\ell+hij\ell)(\gamma+\delta+bhik+ghik)}{\gamma(\gamma+\delta+ahi\ell+bhik+fhik+fhi\ell+ghik+hij\ell+hik\ell)}$
    
    \vspace{0.05in}
    $\displaystyle G=g+\frac{g(\delta+bhik+ghik)}{\gamma}$
    
    \vspace{0.05in}
    $\displaystyle H=\frac{fghjm}{\gamma}$
    
    \vspace{0.05in}
    $\displaystyle I=\frac{fgijm}{\gamma}$
    
    \vspace{0.05in}
    $\displaystyle J=j+\frac{j(\delta+ahi\ell+hij\ell)}{\gamma}$
    
    \vspace{0.05in}
    $\displaystyle K=\frac{k(\gamma+\delta+ahi\ell+hij\ell)}{\gamma+\delta+ahi\ell+bhik+fhik+fhi\ell+ghik+hij\ell+hik\ell}$
    
    \vspace{0.05in}
    $\displaystyle L=\frac{\ell(\delta+bhik+ghik)}{\gamma+\delta+ahi\ell+bhik+fhik+fhi\ell+ghik+hij\ell+hik\ell}$
    
    \vspace{0.05in}
    $\displaystyle M=\frac{ghik\ell}{\gamma+\delta+ahi\ell+bhik+fhik+fhi\ell+ghik+hij\ell+hik\ell}$
    
    \item Conjectured Hexagon Conditional Local Move
    \begin{center}
    \begin{tikzpicture}
    \path[-, line width=0.4mm]
    (0,0) edge node[right]{$a$} (0.75,-1.3)
    (0,0) edge node[right]{$b$} (-0.75,-1.3)
    (0,0) edge node[below]{$c$} (-1.5,0)
    (0,0) edge node[left]{$d$} (-0.75,1.3)
    (0,0) edge node[left]{$e$} (0.75,1.3)
    (0,0) edge node[above]{$f$} (1.5,0)
    (0.75,-1.3) edge node[below]{$g$} (-0.75,-1.3)
    (-0.75,-1.3) edge node[left]{$h$} (-1.5,0)
    (-1.5,0) edge node[left]{$i$} (-0.75,1.3)
    (-0.75,1.3) edge node[above]{$j$} (0.75,1.3)
    (0.75,1.3) edge node[right]{$k$} (1.5,0)
    (1.5,0) edge node[right]{$\ell$} (0.75,-1.3)
    (0.75,-1.3) edge node[right]{$m$} (1.5,-2.6)
    (-0.75,-1.3) edge node[left]{$n$} (-1.5,-2.6)
    (-0.75,1.3) edge node[left]{$o$} (-1.5,2.6)
    (0.75,1.3) edge node[right]{$p$} (1.5,2.6)
    (-1.5,-2.6) edge node[below]{$q$} (1.5,-2.6);
    \draw [line width=0.25mm, fill=black] (0.75,1.3) circle (1.5mm);
    \draw [line width=0.25mm, fill=black] (0.75,-1.3) circle (1.5mm);
    \draw [line width=0.25mm, fill=black] (-0.75,1.3) circle (1.5mm);
    \draw [line width=0.25mm, fill=black] (-0.75,-1.3) circle (1.5mm);
    \draw [line width=0.25mm, fill=gray!60] (1.5,2.6) circle (1.5mm);
    \draw [line width=0.25mm, fill=gray!60] (1.5,-2.6) circle (1.5mm);
    \draw [line width=0.25mm, fill=gray!60] (-1.5,2.6) circle (1.5mm);
    \draw [line width=0.25mm, fill=gray!60] (-1.5,-2.6) circle (1.5mm);
    \draw [line width=0.25mm, fill=gray!60] (-1.5,0) circle (1.5mm);
    \draw [line width=0.25mm, fill=white] (0,0) circle (1.5mm);
    \draw [line width=0.25mm, fill=gray!60] (1.5,0) circle (1.5mm);
    \node at (3.5,0) {\LARGE$\leftrightarrow$};
    \path[-, line width=0.4mm]
    (7,0) edge node[right]{$A$} (7.75,-1.3)
    (7,0) edge node[right]{$B$} (6.25,-1.3)
    (7,0) edge node[below]{$C$} (5.5,0)
    (7,0) edge node[left]{$D$} (6.25,1.3)
    (7,0) edge node[left]{$E$} (7.75,1.3)
    (7,0) edge node[above]{$F$} (8.5,0)
    (7.75,-1.3) edge node[below]{$G$} (6.25,-1.3)
    (6.25,-1.3) edge node[left]{$H$} (5.5,0)
    (5.5,0) edge node[left]{$I$} (6.25,1.3)
    (6.25,1.3) edge node[above]{$J$} (7.75,1.3)
    (7.75,1.3) edge node[right]{$K$} (8.5,0)
    (8.5,0) edge node[right]{$L$} (7.75,-1.3)
    (7.75,-1.3) edge node[right]{$M$} (8.5,-2.6)
    (6.25,-1.3) edge node[left]{$N$} (5.5,-2.6)
    (6.25,1.3) edge node[left]{$O$} (5.5,2.6)
    (7.75,1.3) edge node[right]{$P$} (8.5,2.6)
    (5.5,2.6) edge node[above]{$Q$} (8.5,2.6);
    \draw [line width=0.25mm, fill=black] (7.75,1.3) circle (1.5mm);
    \draw [line width=0.25mm, fill=black] (7.75,-1.3) circle (1.5mm);
    \draw [line width=0.25mm, fill=black] (6.25,1.3) circle (1.5mm);
    \draw [line width=0.25mm, fill=black] (6.25,-1.3) circle (1.5mm);
    \draw [line width=0.25mm, fill=gray!60] (8.5,2.6) circle (1.5mm);
    \draw [line width=0.25mm, fill=gray!60] (8.5,-2.6) circle (1.5mm);
    \draw [line width=0.25mm, fill=gray!60] (5.5,2.6) circle (1.5mm);
    \draw [line width=0.25mm, fill=gray!60] (5.5,-2.6) circle (1.5mm);
    \draw [line width=0.25mm, fill=gray!60] (5.5,0) circle (1.5mm);
    \draw [line width=0.25mm, fill=white] (7,0) circle (1.5mm);
    \draw [line width=0.25mm, fill=gray!60] (8.5,0) circle (1.5mm);
    \end{tikzpicture}
    \end{center}
    In this case the move is symmetric, so we only include formulas for one direction.
    
    Let $\alpha=gmn+gmq+gnq+mnq, \beta=abq+agq+ahq+anq+bgq+b\ell q+bmq+ghq+g\ell q+h\ell q+hmq+\ell nq, \gamma=degh\ell q+dghj\ell q+dghk\ell q+eghi\ell q+eghj\ell q+ghij\ell q+ghik\ell q+ghijk\ell q$, and $\delta=dgh\ell pq+ghi\ell pq+egh\ell oq+ghk\ell oq+ghj\ell oq+ghj\ell pq+ghlopq$.
    
    \vspace{0.05in}
    $\displaystyle A=\frac{a(\alpha+bmq+hmq)}{\alpha+\beta}$
    
    \vspace{0.05in}
    $\displaystyle B=\frac{b(\alpha+anq+\ell nq)}{\alpha+\beta}$
    
    \vspace{0.05in}
    $\displaystyle 
    C=c+\frac{hjk(abq+agq+bgq+b\ell q+bmq)-degh\ell q-dghj\ell q-dghk\ell q-dgh\ell pq-eghj\ell q}{jk(\alpha+\beta)}$
    
    \vspace{0.05in}
    $\displaystyle D=d+\frac{d(\gamma+dgh\ell pq+ghi\ell pq)}{ijk(\alpha+\beta)}$
    
    \vspace{0.05in}
    $\displaystyle E=e+\frac{e(\gamma+egh\ell oq+ghk\ell oq)}{ijk(\alpha+\beta)}$
    
    \vspace{0.05in}
    $\displaystyle F=f+\frac{ij\ell(abq+agq+ahq+anq+bgq)-degh\ell q-dghj\ell q-eghi\ell q-eghj\ell q-egh\ell oq}{ij(\alpha+\beta)}$
    
    \vspace{0.05in}
    $\displaystyle G=\frac{(\alpha+bmq+hmq)(\alpha+anq+\ell nq)}{mn(\alpha+\beta)}$
    
    \vspace{0.05in}
    $\displaystyle H=\frac{h(\alpha+anq+\ell nq)}{\alpha+\beta}$
    
    \vspace{0.05in}
    $\displaystyle I=i+\frac{\gamma+dgh\ell pq+ghi\ell pq}{jk(\alpha+\beta)}$
    
    \vspace{0.05in}
    $\displaystyle J=\frac{(ijk(\alpha+\beta)+\gamma+dgh\ell pq+ghi\ell pq)(ijk(\alpha+\beta)+\gamma+egh\ell oq+ghk\ell oq)}{ik(\alpha+\beta)(ijk(\alpha+\beta)+\gamma+\delta)}$
    
    \vspace{0.05in}
    $\displaystyle K=k+\frac{k(\gamma+egh\ell oq+ghk\ell oq)}{ijk(\alpha+\beta)}$
    
    \vspace{0.05in}
    $\displaystyle L=\frac{\ell(\alpha+bmq+hmq)}{\alpha+\beta}$
    
    \vspace{0.05in}
    $\displaystyle M=\frac{\alpha+bmq+hmq}{gn}$
    
    \vspace{0.05in}
    $\displaystyle N=\frac{\alpha+anq+\ell nq}{gm}$
    
    \vspace{0.05in}
    $\displaystyle O=\frac{o(ijk(\alpha+\beta)+\gamma+dgh\ell pq+ghi\ell pq)}{ijk(\alpha+\beta)+\gamma+\delta}$
    
    \vspace{0.05in}
    $\displaystyle P=\frac{p(ijk(\alpha+\beta)+\gamma+egh\ell oq+ghk\ell oq)}{ijk(\alpha+\beta)+\gamma+\delta}$
    
    \vspace{0.05in}
    $\displaystyle Q=\frac{hgj\ell opq}{ijk(\alpha+\beta)+\gamma+\delta}$
\end{itemize}

\bibliographystyle{amsplain}

\end{document}